\documentclass{amsart}

\usepackage{amsthm}
\usepackage{amssymb}
\usepackage{amsmath}
\usepackage{amsmath,amscd}
\usepackage{epsfig}
\usepackage{mathrsfs}
\usepackage{verbatim}

\newtheorem{theorem}{Theorem}
\newtheorem{corollary}{Corollary}
\newtheorem{lemma}{Lemma}
\newtheorem{proposition}{Proposition}

\newtheorem{definition}{Definition}
\newtheorem{remark}{Remark}

\newcommand{\mb}{\mathbb}
\newcommand{\mc}{\mathcal}
\newcommand{\wt}{\widetilde}
\newcommand{\ol}{\overline}

\newcommand{\wh}{\widehat}

\title{Weighted estimates for the  stability operator of  the helicoid on slowly varying domains}

\author{Stephen J. Kleene}
\address{Department of Mathematics, Brown University, Providence, RI 02906}

\begin{document}
\maketitle
\begin{abstract}
We consider the  poisson problem $\wt{\mc{L}} u  = E$ for the  operator $\wt{\mc{L}} = \Delta_{\mb{R}^2} + 2\cosh^{-2}(s)$ on domains of the form $\Lambda : = \{ (s, z): |s| \leq \ell(z) \}$, where the width $\ell(z)$  of the domain $\Lambda$ varies with $z$.   We prove the existence of solutions satisfying weighted estimates when the source term satisfies certain natural orthogonality conditions, and when  $e^{-\ell}$  is small and  slowly varying. 
\end{abstract}
\section{Introduction}
In this article, we consider the Poisson problem
\begin{align} \label{PoissonProblem}
\wt{\mc{L}} u = E
\end{align}
for the operator $\wt{\mathcal{L}} =  \Delta_{\mb{R}^2} + 2 \cosh^{-2} (s)$ on a domain $\Lambda$ of the form 
\begin{align}\label{LambdaDef}
\Lambda : = \{(s, z): |s| \leq \ell (z), z \in \mb{R} \}.
\end{align}
We prove the  existence of solutions to (\ref{PoissonProblem}) satisfying  weighted estimates when the source term $E$ satisfies certain natural orthogonality conditions (\ref{StrongOrthogonality}) arising from the non-positive spectrum of $\wt{\mc{L}}$, and where the width $\ell$ of the domain and the weight function  are controlled by a  \emph{slowly varying} scale function  (\ref{ZConvexityCondition}). We state a version of our main theorem in simple terms below:

\begin{theorem}\label{Prop:MainTheoremSimple}
There are constants $C > 0$ and $c > 0$  such that: Let $E: \Lambda \rightarrow \mb{R}$ be a locally $C^{0, \alpha}$ function on the domain  $\Lambda$ given by (\ref{LambdaDef}), where  $e^{- \ell}$ is suitably small and  slowly varying.  Assume that $E$ satisfies the locally weighted estimate $\left\| E \right\|_{0, \alpha} \leq \beta e^{- \ell}$ and is  strongly orthogonal to the lower eigenfunctions for the operator $\wt{\mc{L}}$. Then there is a locally $C^{2, \alpha}$ function $u: \Lambda \rightarrow \mb{R}$ solving (\ref{PoissonProblem}) on $\Lambda$ and satisfying
\[
\left\| u \right\|_{2, \alpha} \leq C \beta \ell^{c} e^{- \ell}.
\]
\end{theorem}
 Our primary motivation for recording Theorem \ref{MainInvertibilityStatement} is its applications to constructing examples of singular  minimal laminations using PDE methods (see  \cite{K3}, \cite{K4}). The operator $\wt{\mc{L}}$ we consider here is the stability operator of the helicoid in its standard conformal parametrization up to the factor $\cosh^{-2}(s)$. However,  the methods presented here seem likely to  extend to other  operators and   are thus  of some independent interest. 
\subsection{Precise statement of the main theorem} 

Fix  a  positive function $\lambda: \mb{R} \rightarrow \mb{R}$ satisfying 
\begin{align}\label{ZConvexityCondition}
\left|\lambda' \right| \leq C_0 \lambda^{1 + \epsilon_0}
\end{align}
for constants $C_0 >0$ and $\epsilon_0 > 0$ as well as the  basic smallness condition
\begin{align}\label{BasicSmallness}
\sup \lambda \leq \ol{\lambda}
\end{align}
for a constant $\ol{\lambda}$ to be determined. A function satisfying (\ref{ZConvexityCondition}) is said to be \emph{slowly varying}.  The functions $\ell(z)$ describing the width of the domain $\Lambda$ as a function of $z$ will be given by
\begin{align}\label{Def:EllTau}
\ell = \ell_{\tau} : = \begin{cases}  \frac{\left(r_0\lambda^{- \tau \epsilon_0}\right) -1}{\tau \epsilon_0} & \tau > 0 \\  \\\log(r_0 \lambda) & \tau = 0 \end{cases} 
\end{align}
where $r_0 > 0$ is a fixed constant, and we will assume take $\ell = \ell_{\tau}$ for suitably small  $\tau \in [0, \ol{\tau}]$ for a suitably small constant $\ol{\tau}$. Observe that $\ell_{\tau}$ converges smoothly as $\tau \rightarrow 0$ to $\ell_0$, so the family $\ell_\tau$ in (\ref{Def:EllTau}) is continuous in $\tau$.  Condition (\ref{ZConvexityCondition}) implies that the  ratio $\left|\frac{\lambda(z)}{\lambda(z_0)}\right|$ is uniformly controlled when $|z - z_0| \approx \lambda^{- \epsilon_0}(z_0)$. We assume the weight function $\omega = \omega(z)$ satisfies the slightly more flexible condition
\begin{align}\label{WeightFunctionCondition}
\left|\frac{\omega(z)}{\omega(z_0)} - 1 \right| < \rho_0 < 1, \quad \forall |z - z_0| \leq \lambda^{-3 \epsilon_0/4}(z_0).
\end{align}
For a fixed constant $\rho_0 < 1$. This ensures not only that $\omega = \lambda$ is admissible weight function, but also a slightly more general class of weights derived from $\lambda$ that arise in this article, including those of the form $\ell^{p} \lambda^{q}$ for positive $p$ and $q$.  Finally, we assume the source term $E$ satisfies the orthogonality conditions
\begin{align}\label{StrongOrthogonality}
\int_{-\ell(z)}^{\ell(z)} E(s, z) \cosh^{-1}(s)ds = \int_{-\ell(z)}^{\ell(z)} E(s, z) \tanh(s)ds, \quad z \in \mb{R}.
\end{align}
Any function $E$ satisfying the conditions (\ref{StrongOrthogonality}) is said to be strongly orthogonal to $\cosh^{-1}(s)$ and $\tanh(s)$  on the domain $\Lambda$, which span the space of lower eigenfunctions for the operator $\wt{L} : = \partial_{ss} + 2 \cosh^{-2}(s)$.  More generally, the \emph{density} of a function $f$ in $E$ at $z$ is given by
\[
\int_{- \ell(z)}^{\ell(z)} E(s, z) f(s, z) d s,
\]
and we say that $E$ is strongly orthogonal to $f$ if the density of $f$ in $E$ vanishes for each $z$. 
 Below, $\|f\|_{k, \alpha}$ denotes the localized $C^{k, \alpha}$ holder norm for a function $f: \Lambda \rightarrow \mb{R}$, so
\[
\| f\|_{k, \alpha} (p) : = \| f: C^{k, \alpha} (\Lambda \cap D_p) \|
\]
where $D_p \subset \mb{R}^2$ denotes the unit disk in $\mb{R}^2$ centered at $p$. We can now precisely state our main result:
 
\begin{theorem}\label{MainInvertibilityStatement}
There are constants $C > 0$ and $c > 0$ and $\ol{\tau} > 0$ and $\ol{\lambda} > 0$ such that: Given a scale function $\lambda$ satisfying (\ref{ZConvexityCondition}) and (\ref{BasicSmallness}), a  weight function $\omega$ satisfying (\ref{WeightFunctionCondition}), a domain $\Lambda$ of the form (\ref{LambdaDef}) with $\ell = \ell_{\tau}$ for $\tau \in [0, \ol{\tau}]$,  and  a function $E: \Lambda \rightarrow \mb{R}$ satisfying the locally weighted estimate  $\| E\|_{0, \alpha} \leq \beta \omega$ for a constant $\beta \geq 0$  and  the strong orthogonality conditions (\ref{StrongOrthogonality}), there is a function $u: \Lambda \rightarrow \mb{R}$ with  $\| u\|_{2, \alpha} \leq C \beta \ell^{c}\omega$ and such that 
\[
\wt{\mc{L}} u = E.
\]
\end{theorem}

\subsection{Proof of Theorem \ref{MainInvertibilityStatement}}

Theorem \ref{MainInvertibilityStatement} is  an application of the decay estimate for the Poisson Problem for the operator $\wt{\mc{L}}$ recorded in \cite{K2} which we state here for the convenience of the reader:
 
  \begin{theorem} \label{RBCSolutionsWeightedControl}
For any $\xi \in (0, 1/2)$, there is a constant $C = C(\xi)$ such that: Given a  constant $\ell > 1$ and a $C^{0, \alpha}$ function $E(s, z): [-\ell, \ell] \times \mb{R} \rightarrow \mb{R}$ supported on the strip $[- \ell, \ell] \times [-2 \pi, 2 \pi]$ with $\| E\|_{0, \alpha} \leq \beta$ and satisfying the strong orthogonality conditions (\ref{StrongOrthogonality}), there is a unique locally $C^{2, \alpha}$ function $u: [- \ell, \ell] \times \mb{R}\rightarrow \mb{R}$ solving  (\ref{PoissonProblem}) and satisfying the orthogonality conditions (\ref{StrongOrthogonality}) and such that: 
\begin{enumerate}
\item \label{ODDRBS} The odd part $u^{-}(s ,z) : = \frac{1}{2}\left(u(s, z) - u(-s, z) \right)$ of $u$ satisfies the boundary condition
\[
u^{-}_{, s}(\ell, z) \tanh(\ell) - u^{-}(\ell, z) \tanh'(\ell) = 0, \quad \forall z \in \mb{R}.
\]
\item \label{EVENRBS} The even part $u^{+}(s ,z) : = \frac{1}{2}\left(u(s, z) +u(-s, z) \right)$ of $u$ satisfies the boundary condition
\[
u^{+}_{, s}(\ell, z) \cosh^{-1}(\ell) - u^{+}(\ell, z) \left(\cosh^{-1}\right)'(\ell) = 0, \quad \forall z \in \mb{R}.
\]
\end{enumerate}
\end{theorem}

The general idea of the proof is then as follows: Localize the source term $E$ to  horizontal strips of fixed height, find solutions to the Poisson problem for the localized source terms with decay using Theorem \ref{RBCSolutionsWeightedControl}, and truncate them slowly  over an interval which is large enough for the truncation to produce a definitely smaller error due to the decay of the solutions, and small enough so that solutions at vastly different scales don't interact. If this process results in a strictly smaller error term, it can be iterated to find an exact solution.   Condition (\ref{ZConvexityCondition}) implies  that where the weight function $\lambda$ is small, it remains so over  large domains (this is quantified in Lemmas \ref{ZGoodInterval} and \ref{ZBetterInterval}). This is important for proving  the weighted estimates of Theorem \ref{MainInvertibilityStatement}, since regions of small scale are buffeted by large distances from regions of large  scale.

\subsubsection{The case $\lambda = const$} We  discuss first the proof of Theorem  \ref{MainInvertibilityStatement} in the case that the scale function is constant in order to more carefully illustrate the main ideas of the proof described above.  The discussion will use some notation and terminology that is introduced in Section \ref{Sec:Preliminaries}, to which the reader may refer.   In this case, the domain $\Lambda$ for the linear problem is a  vertical strip $[-\ell, \ell] \times \mb{R}$ for $\ell \approx \log (\lambda)$  and the local $C^{0, \alpha}$ norm $\|E \|_{0, \alpha}$ of the source  $E$ is essentially a fixed constant $\beta \omega$.  Theorem \ref{MainInvertibilityStatement}  follows from Theorem \ref{RBCSolutionsWeightedControl} by localizing the source term to  the rectangles $R_j$ of fixed height using  the partition of unity $\psi_j$, truncating the solutions slowly over the  tall domains $R^*_j$ with heights $h^* \approx \ell^{\sigma}$ for a power $\sigma > 2$ using the cutoff functions $\varphi_j^*$,  and then summing to form a global approximate solution, and showing that the resulting error term  is smaller than the original. An iteration process then finishes the proof. More precisely, we set $E_j (s, z): = \psi_j(z) E(s, z)$. By Theorem \ref{RBCSolutionsWeightedControl} there are  $u_j: [- \ell, \ell] \times \mb{R} \rightarrow \mb{R}$ satisfying
 \[
 \wt{\mc{L}} u_j = E_j
 \]
 and the weighted estimate
 \[
 \left\| u_j \right\|_{2, \alpha} \leq  \ell^{5/2} \beta \omega \left(\frac{1}{1 + |z - z_j|}\right)^{\xi}.
 \]
 The truncated solutions $u_j^* = \psi^*_j u_j$ are compactly supported in the  rectangles $R^*_j$, and it holds
\begin{align*}
\wt{\mc{L}}u^*_j & = \psi^*_j E_j + 2(\psi_j^{*})' u'_j + (\psi_j^{*})^{''} u_j \\
& = E_j + O\left(\ell^{5/2} \beta\omega\frac{\left(h^*\right)^{-1}}{1 + |z - z_j|^{\xi}}\right)
\end{align*}
were above we have used that $E_j$ is supported on the set where $\psi^*_j \equiv 1$, and we have used $'$ to denote derivation in $z$. Summing the localized solutions gives a global approximate solution $u^*$ satisfying 
\[
u^* \approx \beta \ell^{5/2} \omega\left(h^*\right)^{1 -\xi}, \quad \wt{\mc{L}}u^*=  E + O\left(\beta \ell^{5/2}\omega \left(h^*\right)^{ - \xi} \right) = E + E^*
\]
where above we have used that for $|j - k| >  2 h^*$, $u^*_k \equiv 0$ on  $R_j$ and the estimate
\[
\sum_{|k - j| \leq 2 h^*}\frac{1}{|z - z_k|^{\xi}} \leq 2 \int_0^{2 h^*} \frac{1}{(1 + x)^{\xi}} \leq \frac{1}{1 - \xi} (1 + 2 h^*)^{1- \xi}
\]
Thus  we have $u^* \approx \beta  \ell^{5/2 + (1 - \xi )\sigma}\omega$ and  $E^* \approx \beta \ell^{5/2 - \xi \sigma } \omega < <  \delta \beta \omega$ provided $5/2 - \xi \sigma < 0$, and thus the new error term $E^*$ is prescriptably small relative to $E$ . Moreover, since the solutions $u_j$ inherit the strong orthogonality of the source term, so does  $E^*$ and we can iterate the process to obtain an exact solution.  
\subsubsection{Considerations  in the general case}
In the case of a non-constant scale function, the same general idea applies, however there are several additional complications that need to be considered.  Firstly, when $\lambda$ is not constant the domain $\Lambda$ is no longer strip, but rather the boundary is a slowly varying graph over the $z$ axis. This causes some minor problems in the application of Theorem \ref{RBCSolutionsWeightedControl}, where the domains of the solutions are  infinitely tall strips.  The problem is solved by finding solutions that are supported on rectangles $R^*_j$ that are tall enough for the truncation argument to work and wide enough to include the relevant portion of the domain $\Lambda$. This, in turn creates a minor complication due to the fact that the source terms are not compactly supported away from the boundary of  $\Lambda$ and cannot a-priori  be considered $C^{0, \alpha}$ functions on $R^*_j$. This is dealt with by making use of an extension operator that preserves the strong orthogonality needed for  Theorem \ref{RBCSolutionsWeightedControl}, and  we can thus assume that the source terms are supported away from the boundary of the extended domain $\check{\Lambda}$.

 Another complication in the general case is that, since the scale function varies, the truncation and iteration argument described above may fail if solutions supported on regions of vastly different scale interact. Thus, the truncation needs to happen sufficiently slowly so that the iteration produces a smaller error term, but rapidly enough so that only localized solutions at comparable scales interact.  Condition (\ref{ZConvexityCondition}) on the scale functions implies that regions of extremely small scale are buffeted by a large distance from other scales, so that there is enough space for the truncation  argument to produce a smaller error term $E^*$. Precisely, Lemma \ref{ZGoodInterval} shows that the scale function $\lambda$ at points in an interval of the order $\lambda^{- \epsilon_0}(z)$ about $z$ is uniformly comparable to $\lambda(z)$. Since our truncation argument needs rectangles with height roughly $h \approx \ell^{\sigma} < \lambda^{-\epsilon_0/2}$ assuming $\ol{\lambda}$ small, there is enough space to cut off slowly before different scales interact.  
 
 There is also a minor difficulty in the iteration phase of the argument related to the strong orthogonality conditions needed for  Theorem \ref{RBCSolutionsWeightedControl}. Namely, since the solutions  $u_j$ are supported on the thickened rectangles $\check{R}^*_j$, they inherit the strong orthogonality to $\tanh(s)$ and $\cosh^{-1}(s)$ along the lines segments $[-\ol{\ell}_j^*- 1, \ol{\ell}_j^* + 1] \times \{ z\}$. The solutions $u_j$ and thus the error term $E^*$ are not in general strongly orthogonal along the shorter lines segments $[-\ell (z), \ell(z)] \times \{ z\}$ foliating the domain $\Lambda$. Thus, we have to be able to prescribe the kernel density in the source terms by the addition of special functions, and  the densities of $\tanh(s)$ and $\cosh^{-1}(s)$  in  the error term $E^*$ need to be carefully estimated. This is particularly sensitive in the case of $\cosh^{-1}(s)$, due to the  exponential growth of  density prescribing function in this case. 
 
  \subsection{Organization of the article}
Section \ref{Sec:Preliminaries} records basic notation and terminology that are used throughout the article, as well as some basic lemmas. Section \ref{Sec:PrescribinKernelDensities} records in Propositions  \ref{Prop:TanhProjectors} and \ref{Prop:CoshProjectors} the  construction of the functions that are used to prescribed $\tanh(s)$ and $\cosh^{-1}(s)$ density.  In  Section \ref{Sec:MainInvertibilityProof} we construct the kernel density preserving extension operator in Lemma \ref{ExtensionLemma} and we record the proof of  Theorem \ref{MainInvertibilityStatement}, using the iteration argument described above. 

 \section{Preliminaries} \label{Sec:Preliminaries}
In this section we record basic notation and terminology, as well as some  basic Lemmas, that will be used throughout. 
 \subsection{Domains and subdomains}

\begin{enumerate}
\item $\Lambda$ will denote the domain of immersion $F$ where the analysis take place. It it is of the form
\[
\Lambda : = \{ (s, z): |s| \leq \ell(z), z \in \mb{R} \}
\]
for a function $\ell: \mb{R} \rightarrow \mb{R}$ given by  (\ref{Def:EllTau}). \\

\item $\check{\Lambda}$ is the \emph{extended domain}, given by
\[
\check{\Lambda} : = \{ (s, z): |s| \leq \ell(z) + 1, z \in \mb{R} \}
\]
It is needed in the truncation/iteration argument, since the inhomogeneous terms we consider in the linear problem must be extended to $\check{\Lambda}$ with compact support. \\

\item $\Lambda_j$  and $\check{\Lambda}_j$  denote the intersection of $\Lambda$ and $\check{\Lambda}$, respectively,  with the horizontal slab $\mb{R} \times I_j$, where $I_j : = [z_j - 2 \pi, z_j+ 2\pi]$  and $z_j = 2 \pi _j$.   \\

\item $R_j$ amd $\check{R}_j$ denote the smallest rectangle containing $\Lambda_j$ and $\check{\Lambda}_j$, respectively. Setting $\ol{\ell}_j : = \sup_{z \in I_j} \ell(z)$, they  are given by
\[
R_j = [-\ol{\ell}_j, \ol{\ell}_j] \times I_j, \quad  \check{R}_j = [-\ol{\ell}_j - 1, \ol{\ell}_j + 1] \times I_j
\]
The domains $R_j$ are used in the construction of the $\tanh(s)$ density prescribing functions in Section  \ref{Sec:PrescribinKernelDensities} and the domains $\check{R}_j$ are used in the truncation/iteration argument in the proof of Proposition \ref{MainInvertibilityStatement}.\\
\end{enumerate}

The following domains are defined relative to a parameter $\sigma$. We will throughout the article assume that $\sigma$ belongs to an interval $[2, \ol{\sigma}]$ for a constant $\ol{\sigma} > 2$ to be determined. \\

\begin{enumerate}
\item $\Lambda^*_j$ and $R^*_j$:  $\Lambda_j^{*}$ denotes the intersection of $\Lambda$ with the slab $\mb{R} \times I_j^*$,  where
\[
I_j^* : = [z_j- \pi h^*_j, z_j +  \pi h_j^*],
\]
and  where $h^*_j = \left(\ell_j\right)^{\sigma}$ and $\ell_j : = \ell(z_j)$.  $R^*_j$ denotes the smallest rectangle containing $\Lambda^*_j$. With 
 \[
 \ol{\ell}^*_j :  = \sup_{z \in I^*_j} \ell(z)
 \]
  we have:
\[
R^*_j = [- \ol{\ell}^*_j, \ol{\ell}^*_j] \times I^*_j \\
\] 
\item  \emph{$\check{\Lambda}^*_j$ and $\check{R}^*_j$}: These are defined in analogy with $\Lambda^*_j$ and $R^*_j$. Thus, $\check{\Lambda}_j^{*}$ denotes the intersection of $\check{\Lambda}$ with the slab $\mb{R} \times I_j^*$, and 
\[
\check{R}^*_j = [- \ol{\ell}^*_j-1, \ol{\ell}^*_j + 1] \times I^*_j
\] 
These are used in the truncation/iteration argument in the proof of Proposition \ref{MainInvertibilityStatement}.  \\

\item \emph{The sequence of points $z_j^{(\sigma)}$}:  We wish to construct  domains $\Lambda^{(\sigma)}_j$ of variable height that partition $\Lambda$, where the height is a  function of $\lambda$. We first  construct a sequence $z_j^{(\sigma)}$ inductively as follows: Setting $z^{(\sigma)}_0 = 0$, we set $z^{(\sigma)}_{j + 1} = z^{(\sigma)}_j + 2\pi  h^{(\sigma)}_j$, where $h^{(\sigma)}_j =  \left(\ell^{(\sigma)}_j \right)^{\sigma}$ and $\ell^{(\sigma)}_{j} : = \ell\left(z^{(\sigma)}_j\right)$ for $j$ positive and for $j$ negative we similarly put $z^{(\sigma)}_{j} = z^{(\sigma)}_{j + 1} + 2 \pi h^{(\sigma)}_{j + 1}$. \\

\item $\Lambda^{(\sigma)}_j$ and $R^{(\sigma)}_j$: $\Lambda^{(\sigma)}_j$ is  intersection of $\Lambda$ with the slab $\mb{R} \times I^{(\sigma)}_j$, where for $ j \geq 0$
\[
I^{(\sigma)}_{j} =  \left[z^{(\sigma)}_j -  \pi h^{(\sigma)}_{j - 1}, z^{(\sigma)}_j + \pi  h^{(\sigma)}_j\right] .
\]
and for $j < 0$ the definition is similar. It is easy to see that these slabs partition $\mb{R}^2$ and thus the domains $\Lambda^{(\sigma)}_j$ partition $\Lambda$. These are used in the construction of the $\cosh^{-1}(s)$-density prescribing functions where, relative to the case of $\tanh(s)$, greater care needs to be taken due to the exponential growth of the required functions. Similarly we put
\[
R^{(\sigma)}_j : = \left[-\ol{\ell}_j^{(\sigma)}, \ol{\ell}_{j}^{(\sigma)}\right] \times I^{(\sigma)}_{j} 
\]
where $\ol{\ell}^{(\sigma)}_j = \sup_{z \in I^{(\sigma)}_j} \ell(z)$. Thus, $R^{(\sigma)}_j$ is the smallest rectangle containing $\Lambda^{(\sigma)}_j$.
\end{enumerate} 

\subsection{Partitions of unity and bump functions} \label{POFUABF}
\begin{enumerate}
\item $\psi_j$: We fix throughout this article a smooth partition of unity $\{\psi_j\}$ of $\mb{R}$ such that $\psi_{j + 1}(z) = \psi_j(z - 2\pi)$ and such that $\psi_0(z)$ is even in $z$ and supported on the interval $[-3 \pi/2, 3 \pi/2]$. They are used to localize the source terms in the linear problem so that Theorem \ref{RBCSolutionsWeightedControl} can be applied in the iteration/truncation argument. For general $j$ the function $\psi_j$ is supported on the interval $I''_j : = [-3 \pi/2 + z_j, 3 \pi/2 + z_j]$.\\

\item $\varphi_j$: We fix  a family of bump functions $\varphi_j = \varphi_j(z)$ on $\mb{R}$ with $\varphi_{j + 1}(z) = \varphi_j(z - 2 \pi)$ and such that  $\varphi_0(z)$ is even in $z$ and supported on the interval $[-2/3 \pi, 2/3 \pi]$. These can be thought of as having been obtained from the functions $\psi_j$ by shrinking their supports slightly, so
\[
\varphi_0(z) = \psi_0\left(\frac{9}{4}z\right)
\]
The difference between the $\psi_j's$ and $\varphi_j's$ is that the supports of $\varphi_j's$ do not overlap and they are thus not a partition of unity  of $\mb{R}$. They are used in the  construction of the functions prescribing  $\tanh(s)$ density in Section \ref{Sec:PrescribinKernelDensities}.  For general $j$ the function $\varphi_j$ is supported on the interval $I'_j : = [-3 \pi/2 + z_j, 3 \pi/2 + z_j]$. Observe that we have the containments $I'_j \subset I_j \subset I''_j$. \\

\item $\varphi_j^{(\sigma)}$: These are similar to the $\varphi_j's$ defined above, however they are supported over the intervals $I_j^{(\sigma)}$ and their derivative is inversely  proportional to the length of the domain. Precisely, we take
\[
\varphi^{(\sigma)}_j = \varphi_0\left( \frac{z - z_j}{h^{(\sigma)}_j}\right).
\]
They are supported on the interval $I^{(\sigma)'}_j : = \left[-\frac{2 \pi}{3}  h^{(\sigma)}_{j}+ z^{(\sigma)}_j, \frac{2 \pi}{3}  h^{(\sigma)}_{j}+ z^{(\sigma)}_j \right]$ and the $k^{th}$ derivative satisfies the estimate:
\begin{align}\label{VarphiSigmaEst}
\left(\varphi_j^{(\sigma)}\right)^{(k)} \leq^C \left(h^{(\sigma)}_j\right)^{- k} = \left(\ell^{(\sigma)}_{j}\right)^{-k \sigma}.
\end{align}
These functions are used in  the construction of the $\cosh^{-1}(s)$ density prescribing functions.  \\

\item  $\varphi_j^*$: These are cutoff functions with supports in the domains $I_j^*$. So
\[
\varphi^*_j (z) = \left( \frac{z - z_j}{h_j^*}\right)
\]
 These are used in the Proof of Theorem \ref{MainInvertibilityStatement} during the   truncation/iteration argument. The solutions $v_j$ to the localized linear problem decay away from the support of the source term and are cutoff off over the tall domains $R_j^*$ using the functions $\psi^*_j$.
\end{enumerate}
\subsection{Properties of scale functions}

The convexity condition (\ref{ZConvexityCondition}) ensures uniform comparability of scales on large domains, which is used  in several places and which we record here for easy reference.
 
 \begin{lemma}\label{ZGoodInterval}
Given $w$ and $z$ in $\mb{R}$ with   
\begin{align} \notag
|z - w| \leq \frac{1}{2 C_0} \left(\frac{2}{3} \right)^{1 + \epsilon_0}  \lambda^{- \epsilon_0}(z)
\end{align}
it holds that 
\begin{align}
\left|\frac{\lambda(w)}{\lambda(z)} - 1 \right| \leq 1/2.
\end{align}
\end{lemma} 
\begin{proof}
For $\delta > 0$ to be determined, suppose there is a point  $w$ with$|z - w| \leq  \delta \lambda^{- \epsilon_0}(z)$ and  $\left|\frac{\lambda(w)}{\lambda(z)} - 1 \right| =1/2$. We pick the point closest to $z$ such that these conditions are satisfied.
\begin{align*}
\left|\lambda(z) - \lambda(w)\right| & = \left| \int_{w}^z \lambda' (z') dz'\right| \\
& \leq C_0\int_{w}^z \lambda^{1 + \epsilon_0}(z') dz' \\
&  \leq C_0 \lambda^{1 + \epsilon_0}(z) \int_{w}^z \left(\frac{\lambda(z')}{\lambda(z)}\right)^{1 + \epsilon_0} dz' \\
& \leq C_0 \lambda^{1+ \epsilon_0}(z)\left(\frac{3}{2}\right)^{1 + \epsilon_0}|z - w| \\
& \leq C_0\delta \lambda(z)\left(\frac{3}{2}\right)^{1 + \epsilon_0} 
\end{align*}
Dividing by $\lambda(z)$ gives
\[
\frac{1}{2} = \left|\frac{\lambda(w)}{\lambda(z)} - 1\right| \leq C_0  \delta \left(\frac{3}{2}\right)^{1 + \epsilon_0}.
\]
Taking $\delta$ small gives a contradiction. 
\end{proof}

As a corollary, we get improved  control on $\lambda$ at slightly smaller scales
 
 \begin{lemma} \label{ZBetterInterval}
 Assume that $|z - w| \leq A \lambda^{- \eta \epsilon_0}$ for constant $A  > 0$ and $\eta \in (0, 1)$. Then for 
 \[
 \sup \lambda \leq \left(\frac{1}{2 C_0 A} \left(\frac{2}{3} \right)^{1 + \epsilon_0}   \right)^{\frac{1}{(1 - \eta) \epsilon_0}}
 \]
 it holds that 
 \begin{align*}
 \left|\frac{\lambda(w)}{\lambda(z)} - 1 \right| \leq C A \lambda^{(1- \eta)\epsilon_0}(z) 
 \end{align*}
 \end{lemma}
 \begin{proof}
Taking $\sup \lambda$ as  in the hypothesis gives
\begin{align*}
|z - w| &  \leq A \lambda^{- \eta \epsilon_0} \\
&  = A \lambda^{ (1 - \eta) \epsilon_0} \lambda^{- \epsilon_0} \\
& \leq A \frac{1}{2 C_0 A} \left(\frac{2}{3} \right)^{1 + \epsilon_0}  \lambda^{- \epsilon_0} \\
& =  \frac{1}{2 C_0 } \left(\frac{2}{3} \right)^{1 + \epsilon_0}  \lambda^{- \epsilon_0}.
\end{align*}
 By Lemma \ref{ZGoodInterval} we then have
\begin{align*}
\left|\lambda(z) - \lambda(w) \right| & = \left| \int_{w}^z \lambda'(z')dz' \right| \\
& \leq C_0 \int_{w}^z \lambda^{1 + \epsilon_0}(z')dz' \\
& \leq C_0 \left(\frac{3}{2}\right)^{1 + \epsilon_0} \lambda^{1 + \epsilon_0}(z) |z - w| \\
& \leq C_0\left(\frac{3}{2}\right)^{1 + \epsilon_0} A \lambda^{1 + (1- \eta)\epsilon_0}(z) .
\end{align*}
Dividing by $\lambda(z)$ gives the claim.
 \end{proof}
 
 By Lemma \ref{ZGoodInterval} and Lemma \ref{ZBetterInterval} we get quantitative control  on the variation of $\lambda$ over intervals of the order $\lambda^{- \eta \epsilon_0 }$, provided $\ol{\lambda}$ is chosen small relative to the constants $C_0$, $\epsilon_0$,   and $A$. Throughout the article, we will assume $\ol{\lambda}$ is chosen small relative to these constraints as needed. 
 
 As a consequence of  Lemma \ref{ZBetterInterval}    we can characterize when the domains $R^*_j$ and $R_k$ intersect.
 
 \begin{lemma}\label{Prop:EllBounds}
 Given  $\delta > 0$ and $\eta \in (0, 1)$  there is $\ol{\lambda} = \ol{\lambda} ( \delta, \eta) >0$ such that: Given $\tau > 0$ such that $\tau  < \frac{\eta}{2\ol{\sigma}}$ it holds that 
 \[
 \ell_{\tau}^\sigma \lambda^{\eta \epsilon_0} < \delta
 \]
 \end{lemma}
 \begin{proof}
Assuming $r_0\lambda < 1$ we have
 \begin{align*}
 \left(r_0\lambda\right)^{- p} & = \left(\left(r_0\lambda\right)^{-1}\right)^{p} = 1 + \int_{0}^p \frac{d}{dp} \left. \left(\left(r_0\lambda^{-1}\right) \right)^{p} \right|_{p = p'} dp' \\
 & = 1 - \log\left(r_0 \lambda\right) \int_{0}^p \left(\left(r_0 \lambda \right)^{-1} \right)^{p'} dp'
 \end{align*}
 Thus we have 
 \begin{align*}
  \frac{\left(r_0\lambda\right)^{- p} - 1}{p} & = - \frac{ \log\left(r_0 \lambda\right) }{p}\int_{0}^p \left(\left(r_0 \lambda \right)^{-1} \right)^{p'} dp' \\
  & \leq - \frac{ \log\left(r_0 \lambda\right) }{p}\int_{0}^p \left(\left(r_0 \lambda \right)^{-1} \right)^{p} dp' \\
  & \leq - \log\left(r_0 \lambda\right) \left(\left(r_0 \lambda \right)^{-1} \right)^{p} 
\end{align*}
where in the second to last inequality above we have used that $\left(r_0\lambda\right)^{-1} > 1$ so that $\left(\left(r_0\lambda\right)^{-1}\right)^{p}$ is monotonically increasing in $p$. Thus, we have
\begin{align*}
\ell^{\sigma}_{\tau} \leq - \log^{\sigma}(r_0 \lambda)\left(r_0 \lambda \right)^{- \sigma \tau \epsilon_0} 
\end{align*}
and
\begin{align*}
\lambda^{\eta\epsilon_0}\ell^{\sigma}_{\tau}  & \leq - \log^{\sigma}(r_0 \lambda)\left(r_0 \lambda \right)^{(\eta- \sigma \tau) \epsilon_0}  \\
& \leq  - \log^{\ol{\sigma}}(r_0 \lambda)\left(r_0 \lambda \right)^{\frac{\eta}{4} \epsilon_0}.
\end{align*}
The claim then follows directly. 
 \end{proof}
 
 \begin{corollary}\label{Prop:DomainIntersection}
 Assume $ \tau  < \frac{1}{4\ol{\sigma}}$. Then there is  $\ol{\lambda}$  such that: Suppose the domain $R^*_{j}$ intersects $R_k$. Then $|j - k|  \leq 2 h^*_{k}$.
 \end{corollary}

\begin{proof}
Assume that $\ol{\lambda}$ is as in the statement of Lemma \ref{Prop:EllBounds} with $\delta =1$. Observe that $R^*_{j}$ intersects $R_k$ if and only if $I^*_j$ intersects $I_k$. Suppose there is $w$ such that  belongs to both  $I^*_j$ and $I_k$. Then we have $|w - z_k| <  \pi$ and $|w- z_j| \leq \pi h^*_j$. Thus we have
\begin{align}\label{Rasdfsk}
\left|z_k - z_j \right| & \leq \pi + \pi  h^*_j \\ \notag
 & \leq 2\pi h^*_j \\ \notag
& \leq 2 \pi\left(\ell_{j}\right)^{\sigma} \\ \notag
& \leq 2 \pi \lambda^{- \epsilon_0/2}, \notag
\end{align}
where the last inequality above follows from Lemma \ref{Prop:EllBounds} with $\delta = 1$. We can apply  Lemma \ref{ZBetterInterval} with $\eta = 1/2$ to get
 \begin{align}\label{RatioControled}
 \left| \frac{\lambda(z_k)}{\lambda(z_j)} - 1\right|  & \leq^C \lambda^{\epsilon_0/2}(z_j)    \\ \notag
 \end{align}
by assuming $\ol{\lambda}$ smaller if necessary. Observe that $\ell = \ell_{\tau}$ satisfies the differential inequality
 \begin{align}\label{EllDIFFIN}
 \left|\ell'\right| \leq C_0 r_0^{- \tau \epsilon_0} \lambda^{(1 - \tau) \epsilon_0}.
 \end{align}
 We then have
 \begin{align*}
\left|\ell_{k} - \ell_j \right| &  = \left| \int_{z_j}^{z_k}  \ell' (z) dz\right|\\
& \leq C_0 r_0^{- \tau \epsilon_0} \int_{z_j}^{z_k}   \lambda^{(1 - \tau) \epsilon_0}(z) dz \\
& \leq C_0 r_0^{- \tau \epsilon_0} h^*_{j} \left(1 + \lambda_j^{\epsilon_0/2}\right)\lambda_j^{\epsilon_0/2}dz.
 \end{align*}
 Dividing by $\ell_j$ gives
 \begin{align*}
 \left| \frac{\ell_k}{\ell_j}   -1 \right|  & \leq  C_0 r_0^{- \tau \epsilon_0} \left(\ell_{j} \right)^{\sigma - 1}\left(1 + \lambda_j^{\epsilon_0/2}\right)\lambda_j^{\epsilon_0/2} \\
 & \leq C\left(\ell_{j} \right)^{\sigma - 1}\left(1 + \lambda_j^{\epsilon_0/2}\right)\lambda_j^{\epsilon_0/2}
 \end{align*}
 In particular, given $q > 0$ and  assuming $\ol{\lambda}$ is smaller if necessary we have
 \[
\left| \frac{\ell_k}{\ell_j}   -1 \right| \leq  q,
 \]
 and thus
 \[
 (1 - q)^{\sigma} \leq \frac{h^*_k}{h^*_j} \leq (1 + q)^{\sigma}
 \]
  Combing this with  (\ref{Rasdfsk}) we then have  the bound
\[
\left|z_k - z_j \right| \leq 2\pi (1 + q)^{\sigma} h^*_k
\]
 Choosing $q$ so that $q < \left(\frac{3}{2} \right)^{1/\ol{\sigma}} - 1$ gives
 \[
 \left| k - j\right| < \frac{3}{2} h_k^*
 \]
 which directly implies the claim. 
\end{proof}

We will also need

\begin{corollary}
There is $\ol{\lambda}$ sufficiently small, the interval $I^{(\sigma)'}_j : = \left[-\frac{2}{3} \pi h^{(\sigma)}_j + z_j^{(\sigma)}, \frac{2}{3} \pi h^{(\sigma)}_j + z_j^{(\sigma)}\right]$ is compactly contained in the interval $I^{(\sigma)}_j$. Thus, the supports of the functions $\varphi^{(\sigma)}_j$ do not intersect. 
\end{corollary}

\begin{proof}
Assuming $\ol{\lambda}$ is as in Lemma \ref{Prop:EllBounds} with $\delta = 1$ we have
\begin{align*}
\left| z^{(\sigma)}_{j} - z_{j -1}^{(\sigma)}\right| &   = 2 \pi h^{(\sigma)}_{j - 1}  \\
& \leq 2 \pi\left( \lambda^{(\sigma)}_j\right)^{-\epsilon_0/2}.
\end{align*}
Thus as in the proof of Corollary \ref{Prop:DomainIntersection} we can choose $\ol{\lambda} = \ol{\lambda}( \ol{\sigma})$ sufficiently small so that
\begin{align*}
\frac{1}{2} \leq \frac{h^{(\sigma)}_j}{h^{(\sigma)}_{j - 1}} \leq  \frac{3}{2},
\end{align*}
which gives
\[
- \frac{4}{3} h^{(\sigma)}_{j - 1} \leq - \frac{8}{9} h^{(\sigma)}_{j} < -\frac{2}{3} h^{(\sigma)}_{j}
\]
This then gives
\begin{align*}
z^{(\sigma)}_{j -1} + \frac{2}{3} \pi h^{(\sigma)}_{j - 1} & = z^{(\sigma)}_{j} - 2 \pi h^{(\sigma)}_{j - 1} + \frac{2}{3} \pi h^{(\sigma)}_{j -1} \\ \notag
& = z^{(\sigma)}_{j} - \frac{4}{3} \pi h^{(\sigma)}_{j - 1} \\ \notag \\
& < z^{(\sigma)}_{j} - \frac{2}{3} \pi h^{(\sigma)}_{j }
\end{align*}
Thus the intervals $I'_{j - 1}$ and $I'_{j}$ do not intersect.
 \end{proof}

\begin{remark} \label{EllComparable}
In light of the derivative estimate (\ref{EllDIFFIN}) for $\ell_\tau$, the  ratios $\frac{\ell(z)}{\ell(w)}$ are uniformly controlled within the intervals $I_{j}$, $I_j^*$ and $I^{(\sigma)}_j$.   Additionally, by assumption the ratios $\frac{\omega(z)}{\omega(w)}$ are uniformly controlled over these intervals as well. 
\end{remark}

 \section{Prescribing kernel density} \label{Sec:PrescribinKernelDensities}
 This section records  the proofs of Propositions \ref{Prop:TanhProjectors} and  \ref{Prop:CoshProjectors}, which construct  functions that prescribe  $\cosh^{-1}(s)$ and $\tanh(s)$ densities along horizontal lines $ \{z = const \}$. The proofs are similar but are recorded separately as the exponential growth of the functions prescribing $\cosh^{-1}(s)$ requires a slightly more careful treatment.

\subsection{Prescribing $\tanh(s)$ density}

\begin{proposition} \label{Prop:TanhProjectors}
There is a constant $C > 0$ such that: Given $e: \mb{R} \rightarrow \mb{R}$ with $\| e\|_{0, \alpha} \leq \beta \omega$ for a positive constant $\beta$, there is a function $\Theta: \Lambda \rightarrow \mb{R}$ such that  
\[
\int_{-\ell(z)}^{\ell(z)} \left(\wt{\mc{L}}\Theta \right)\tanh(s) ds = e(z) \int_{-\ell(z)}^{\ell(z)} \tanh^2(s)  ds
\]
and satisfying the estimate
\[
\left\|\Theta \right\|_{2, \alpha} \leq^C \beta  \ell^{2}  \omega.
\]
\end{proposition}

The proof of Proposition \ref{Prop:TanhProjectors} proceeds in two steps. Fixing a $\tanh(s)$ density function $e(z)$, we first construct  functions with the same density over the  subdomains $\Lambda''_j$ of fixed height. This reduces the problem to density functions that are orthogonal to $\tanh(s)$ over such domains.  Once this is done, we can directly integrate the remaining  density problem as an ODE in  $z$. The orthogonality to $\tanh(s)$ over the domains $\Lambda''_j$ gives  that the ODE is compactly supported in $z$ and implies uniform estimates for the solutions. The proof of Proposition \ref{Prop:CoshProjectors} in the ensuing section is essentially the same, however greater care needs to be taken in several of steps due to the exponential growth of the  functions  needed for arranging  weak orthogonality over the domains $\Lambda_j$.

\begin{lemma} \label{Prop:TProjectors}
The following statements hold:
\begin{enumerate}
\item \label{Prop:TProjectors1}$\wt{\mc{L}} f = \tanh(s)$, where $f = \frac{1}{2} s^2 \tanh (s) - s$. \\
\item\label{Prop:TProjectors2} $\int_{R_j} \wt{\mc{L}}(\varphi_j f) \tanh(s) = \int_{R_j}\varphi_j \tanh^2(s) ds$.  \\
\item \label{Prop:TProjectors3} $\int_{R_j} \wt{\mc{L}}( (z - z_j) \varphi_j f) (z - z_j) \tanh(s)   = \int_{R_j}\varphi_j  (z - z_j)^2 \tanh^2(s) ds$. 
\end{enumerate}
\end{lemma}
\begin{proof}
(\ref{Prop:TProjectors1}) is an easy computation. For (\ref{Prop:TProjectors3}), without loss of generality assume $j = 0$ so that $z_j = 0$.  Then we have
\begin{align*}
\int_{R_j} \wt{\mc{L}}(\varphi_j  zf) z \tanh (s)&=  \int_{R_j}\left( \varphi_j z \tanh(s)  + (\varphi_j z)'' f\right) z \tanh(s) \\
& =  \int_{R_j }\varphi_j z^2\tanh^2(s) +   \int_{R_j}  (\varphi_j z)'' f z \tanh(s). \\
\end{align*}
Then
\begin{align*}
 \int_{R_j}  (\varphi_j z)'' f z \tanh(s) & = \int_{- 2\pi}^{2 \pi}   (\varphi_j z)''  z dz \int_{-\ol{\ell}_j}^{\ol{\ell}_j} f(s) \tanh (s)ds \\
 & = - \int_{- 2\pi}^{2 \pi}   (\varphi_j z)'  dz \int_{-\ol{\ell}_j}^{\ol{\ell}_j} f(s) \tanh(s) ds \\
 & = 0.
\end{align*}
This gives the claim. The proof of  (\ref{Prop:TProjectors2}) is similar.

\end{proof}

\begin{definition}  \label{Def:TAJBJs}
With  $E(s, z) = e (z)\tanh(s)$ and $E_j : = \psi_j E = e_j \tanh(s)$, we set
\[
a_j : = \int_{\Lambda''_j} E_j(s, z)\tanh(s),  \quad b_{j} : =  \int_{\Lambda''_j} E_j(s, z) (z - z_j) \tanh(s).
\]
where $\Lambda_j'' : = \Lambda \cap \left(I_j'' \times \mb{R} \right)$.
\end{definition}

\begin{lemma} \label{Prop:AJBJs}
It holds that 
\[
|a_j|, |b_j|\leq^C \beta \omega_j \ell_j 
\]
where $\omega_j : = \omega (z_j)$.
\end{lemma}

\begin{proof}
We have:
\begin{align*}
a_j  & = \int_{\Lambda''_j}  E_j\tanh(s) \\
& \leq^C \| e_j\|_{2, \alpha}\left| \Lambda''_j\right| \\
& \leq^C \beta \omega_j \ol{\ell}_j \\
& \leq^C \beta \omega_j \ell_j.
\end{align*}
Observe that in the last line above we have replaced $\ol{\ell}_j$ with $\ell_j$ (See Remark \ref{EllComparable}).
\end{proof}

\begin{lemma} \label{Prop:TAJBJs}
There is a function $\wh{\Theta}_j: \mb{R}^2 \rightarrow \mb{R}$ supported on $I'_j \times \mb{R}$  satisfying the estimates
\[
\left\|\wh{\Theta}_j \right\|_{2, \alpha} \leq^C \beta \omega_j \left( 1 + s^2 \right)
\]
and such that 
\[
\int_{\Lambda_j} \left(\wt{\mc{L}} \wh{\Theta}_j\right) \tanh(s) = a_j, \quad \int_{\Lambda_j}\left( \wt{\mc{L}} \wh{\Theta}_j\right)\left(z - z_j \right) \tanh(s) = b_j.
\]
\end{lemma}

\begin{proof}
Define 
\[
\wh{\Theta}_j = A_j \varphi_j f + B_j \varphi_j \left(z_j - z \right) f
\]
for constants $A_j$ and $B_j$ to be determined and where $f$ is as in Lemma \ref{Prop:TProjectors} (\ref{Prop:TProjectors1}).  As an initial choice we take
\[
A_j = a_j/\int_{R_j}\varphi_j \tanh^2(s), \quad B_j = b_j/\int_{R_j}\varphi_j \left( z - z_j \right)^2\tanh^2(s).
\]
We then have
\begin{align}\label{AJBJBOUNDS}
\left|A_j \right| \leq C \ell_j^{-1} \left|a_j \right|, \quad \left|B_j \right| \leq C \ell_j^{-1} \left|b_j\right|
\end{align}
and thus by Lemma \ref{Prop:AJBJs} the function $\wh{\Theta}_j$ satisfies the bound in the statement of the lemma.  
By Lemma \ref{Prop:TProjectors} (\ref{Prop:TProjectors2}) and using that the two terms defining $\wh{\Theta}_j$ are odd and even with respect to reflections through the line $z = z_j$, respectively, we have
\[
\int_{R_j} \wt{\mc{L}}\wh{\Theta}_j \tanh(s) =  \left(a_j/\int_{R_j}\varphi_j \tanh^2(s)\right)\int_{R_j} \wt{\mc{L}}\left(\varphi_j f \right) \tanh(s) = a_j
\]
and similarly
\[
\int_{R_j} \wt{\mc{L}}\wh{\Theta}_j (z - z_j)\tanh(s)  = b_j.
\]
Set 
\begin{align*}
a_j' =  \int_{R_j\setminus \Lambda_j} \left(\wt{\mc{L}}\wh{\Theta}_j\right) \tanh(s), \quad b_j'  =    \int_{R_j\setminus \Lambda_j} \left(\wt{\mc{L}}\wh{\Theta}_j \right) (z_j - z)\tanh(s).
\end{align*}
Using  the derivative estimate (\ref{EllDIFFIN}) for $\ell_{\tau}$ we have
\begin{align*}
\left|a_j'\right|  & \leq^C \sup_{z \in I_j}\left| \ell(z) - \ol{\ell}_j\right|\left\| \wh{\Theta}_j \right\|_{2, \alpha} \\
&  \leq^C \lambda_j^{ (1 - \tau) \epsilon_0}\left(\left| A_j\right| + \left| B_j \right|\right) (1 + s^2) \\
& \leq^C \lambda_j^{(1 - \tau) \epsilon_0}\left(\left( \left| a_j\right| + \left| b_j\right|\right)/\ol{\ell}_j \right)\left(\ol{\ell}_j\right)^2 \\
& \leq^C \lambda_j^{  (1-  \tau)\epsilon_0} \ol{\ell}_j \left(\left| a_j\right| + (\left| b_j\right| \right).
\end{align*}
and similarly, 
\begin{align} \label{bjStarEst}
\left|b_j'\right|  \leq^C \lambda_j^{ (1 -  \tau)\epsilon_0} \ol{\ell}_j\left( \left| a_j\right| + \left| b_j\right| \right).
\end{align}
Lemma \ref{Prop:EllBounds} then gives
\begin{align*}
 \lambda_j^{ (1 -  \tau)\epsilon_0} \ol{\ell}_j & \leq C  \lambda_j^{ \epsilon_0/2} \ell_j   \\
 & < \delta
 \end{align*}
 by taking $\ol{\lambda}$  small in terms of $\delta$ and assuming $\ol{\tau} < 1/2$. In particular, we have
 \[
\left|a'_j \right| \leq \delta \left(\left| a_j\right| + \left| b_j\right|\right), \quad \left|b'_j \right| \leq \delta\left( \left| a_j\right| +  \left| b_j\right| \right),
\]
for any $\delta > 0$.
An exact solution satisfying the claimed estimates can then be found by iteration and taking $\delta$ sufficiently small. 
\end{proof}
\begin{definition} \label{Def:HAtSolution}
We set
\[
\wh{E}_j : = E_j - \mc{\wt{L}}\left(\wh{\Theta}_j\right).
\]
\end{definition}

\begin{lemma} \label{Prop:HatSolution}
The following statements hold
\begin{enumerate}
\item  \label{Prop:HatSolution2} $\left\| \wh{E}_j\right\|_{0, \alpha}\leq^C\beta  \omega_j (1 + s^2)$. \\
\item \label{Prop:HatSolution3} It holds that 
\[
\int_{\Lambda''_j} \wh{E}_j \tanh (s)= \int_{\Lambda''_j} \wh{E}_j  \left(z - z_j \right) \tanh(s) = 0.
\]
\end{enumerate}
\end{lemma}
\begin{proof}
(\ref{Prop:HatSolution2}) is an immediate consequence of the estimate for $ \wh{\Theta}_j$ in Lemma \ref{Prop:TAJBJs} and the definition of $E$. (\ref{Prop:HatSolution3}) follows directly from the definition of $\wh{E}_j$.
\end{proof}

\begin{definition} \label{Def:BarSolutions}
We set
\begin{enumerate}
\item  \label{Def:BarSolutions1}
\[
\theta_j (z) : = \frac{\int_{-\ell(z)}^{\ell(z)} \wh{E}_j(s, z)  \tanh(s)ds}{\int_{-\ell_j}^{\ell_j} \tanh^2(s) ds}
\]
\item  \label{Def:BarSolutions2}
\begin{align*}
\ol{\Theta}_j(z)  & : = \left(z  - z_j\right)\int_{-\infty}^{z}\theta_j(z') dz' - \int_{-\infty}^{z_j}  \left(z' - z_j\right)\theta_j(z') dz'
\end{align*}

\end{enumerate}
\end{definition}

\begin{lemma} \label{Prop:BarSolutions}
The functions $\theta$ and $\ol{\Theta}$ are supported on $I''_{j}$ and it holds that 
\begin{enumerate}
\item \label{Prop:BarSolutions1} $\left\|\theta_j\right\|_{0, \alpha} \leq^C \beta \ell_j^2\omega_j$.  \\
\item \label{Prop:BarSolutions2} $\left\|\ol{\Theta}_j \right\|_{2, \alpha} \leq^C \beta \ell_j^{2}  \omega_j$.
\end{enumerate}
\end{lemma}

\begin{proof}
From Proposition \ref{Prop:HatSolution} (\ref{Prop:HatSolution2}) we have
\begin{align*}
\left\|\theta_j  \right\|_{0, \alpha}  & \leq  \left\| \wh{E}_j(s, z)\right\|_{0, \alpha} \\
& \leq^C \beta \ell_j^2 \omega_j.   \\
\end{align*}
which is (\ref{Prop:BarSolutions1}). Moreover, $\theta_j$ is clearly supported in $I''_j$ since $\wh{E}_j$ is.  For (\ref{Prop:BarSolutions2}). Observe that  the orthogonality conditions on $\wh{E}$ in Lemma \ref{Prop:HatSolution} (\ref{Prop:HatSolution3}) give
\begin{align*}
 \int_{I''_j} \theta_j(z')  dz'   & = \frac{1}{\int_{-\ell_j}^{\ell_j} \tanh^2(s) ds} \int_{I''_j} \int_{- \ell(z')}^{\ell(z')} \wh{E}(s, z') \tanh(s) ds\, dz'\\
 & =   \frac{1}{\int_{-\ell_j}^{\ell_j} \tanh^2(s) ds} \int_{\Lambda''_j} \wh{E}_j(s, z) \tanh(s)\\ 
 & = 0.
 \end{align*}
A similar argument gives
 \begin{align*}
  \int_{I''_j} \left(z - z_j\right) \theta(z')  dz'  & = \frac{1}{\int_{-\ell_j}^{\ell_j} \tanh^2(s) ds} \int_{\Lambda''_j} \wh{E}(s, z) (z_j - z)\tanh(s)\\
  &  = 0.
  \end{align*}
  Thus,  $\ol{\Theta}_j$  is supported on  $I_j''$. For  $z \in I''_j$ we have
\begin{align*}
\left\|\ol{\Theta}_j(z)\right\|_{2, \alpha} & = \left\|\left(z  - z_j\right) \int_{z_j - 3\pi/2}^z \theta(z') dz' - \int_{z_j - 3\pi/2}^{z} (z' - z_j) \theta(z') dz'  \right\|_{2, \alpha} \\
& \leq^C\left\| \theta \right\|_{0, \alpha} \\
& \leq^C \beta \ell_j^2\omega_j. 
\end{align*}
where the last line above follows from  (\ref{Prop:BarSolutions1}).
\end{proof}

\begin{proof}[Proof of Proposition \ref{Prop:TanhProjectors}]
We set $\Theta_j =  \ol{\Theta}_j \tanh(s) + \wh{\Theta}_j$ and we put $\Theta = \sum_j \Theta_j$. Since $\wt{\mc{L}}\tanh(s) = 0$, we have
\[
\mc{\wt{L}}\left( \ol{\Theta}_j\tanh(s)\right) = \ol{\Theta}_j'' \tanh(s) = \theta_j \tanh(s).
\]
and thus by constuction
\begin{align*}
\int_{-\ell(z)}^{\ell(z)} \left(\wt{\mc{L}}\Theta_j \right)\tanh(s) ds & = \int_{-\ell(z)}^{\ell(z)} \left(\wt{\mc{L}}\wh{\Theta}_j\right)\tanh(s) ds +   \int_{-\ell(z)}^{\ell(z)} \wt{\mc{L}}\left(\ol{\Theta}_j\tanh(s))\right) \tanh(s) ds \\
& = \int_{-\ell(z)}^{\ell(z)} \left(E_j - \wh{E}_j\right) \tanh(s) ds + \int_{-\ell(z)}^{\ell(z)}  \theta_j \tanh^{2}(s) ds \\
& =  \int_{-\ell(z)}^{\ell(z)} \left(E_j - \wh{E}_j\right) \tanh(s) ds  + \left(\frac{\int_{- \ell(z)}^{\ell(z)} \tanh^2(s)ds}{\int_{- \ell_j}^{\ell_j} \tanh^2(s)ds} \right) \int_{-\ell(z)}^{\ell(z)} \wh{E}_j \tanh(s)ds \\
& = e_j(z)\int_{-\ell(z)}^{\ell(z)} \tanh^{2}(s) ds + \left(\frac{\int_{- \ell(z)}^{\ell(z)} \tanh^2(s)ds}{\int_{- \ell_j}^{\ell_j} \tanh^2(s)ds} -1\right) \int_{-\ell(z)}^{\ell(z)} \wh{E}_j \tanh(s)ds \\
& =  e_j(z)\int_{-\ell(z)}^{\ell(z)} \tanh^{2}(s) ds +e'_j (z)\int_{-\ell(z)}^{\ell(z)} \tanh^{2}(s) ds
\end{align*}
where $e'_j (z)$ is defined  implicitly by the last equality above and where we have set $e_j : = \psi_j e$. We have
\begin{align*}
\left\|\int_{- \ell(z)}^{\ell(z)} \tanh^2(s)ds - \int_{- \ell_j}^{\ell_j} \tanh^2(s)ds \right\|_{0, \alpha} & \leq C \sup_{z \in I''_j} \left\| \ell(z) - \ell_j\right\|_{0, \alpha}\\
& \leq C \lambda_j^{(1 - \tau) \epsilon_0}
\end{align*}
and thus
´\begin{align*}
\left\|e'_j \right\|_{0, \alpha} & \leq C \left(\ell_j^{-2}  \lambda_j^{(1 - \tau) \epsilon_0}\right) \left(\beta \ell_j^3  \omega_j \right) \\
& \leq C \beta \ell_j  \lambda_j^{(1 - \tau) \epsilon_0}\omega_j.
\end{align*}
By Lemma \ref{Prop:EllBounds} we can take $\ol{\lambda}$ small such that 
\[
\left\|e'_j \right\|_{0, \alpha} \leq \delta \beta \omega_j
\]
for arbitrary $\delta > 0$. Since the supports of $e'_j$ are contained in $I''_j$,  by Remark \ref{EllComparable}  the sum $e' : = \sum_j e'_j$  then satisfies
\[
\left\|e' \right\|_{0, \alpha} \leq C \delta \beta \omega.
\]
 We can then iterate to find an exact solution $\Theta$ satisfying the claimed bounds.
 \end{proof}

\subsection{Prescribing strong $\cosh^{-1}(s)$ orthogonality}
\begin{proposition} \label{Prop:CoshProjectors}
There is a  constant $C >0$  such that: Given $e:\mb{R} \rightarrow \mb{R}$ with $\| e\|_{0, \alpha} \leq \beta \omega$ for a positive constant  $\beta$, there is a function $\Theta: \Lambda \rightarrow \mb{R}$ such that  
\[
\int_{-\ell}^{\ell} \left(\wt{\mc{L}}\Theta \right)\cosh^{-1}(s) ds = e (z)\int_{-\ell}^\ell \cosh^{-2}(s) ds
\]
and satisfying the estimates
\[
\left\|\Theta \right\|_{2, \alpha} \leq^C \beta \omega \cosh(s)
\]
and
\[
\left\|\wt{\mc{L}} \Theta\right\|_{2, \alpha} \leq^C \beta \omega(\cosh^{-1}(s) + \ell^{- \sigma}\cosh(s)).
\]
\end{proposition}

\begin{lemma}\label{Prop:PartitionSigma}
There is a  smooth partition of unity $\psi^{(\sigma)}_j(z)$ subordinate to the covering 
\[
I^{(\sigma)''}_j : = \left[z_j^{(\sigma)} - 3/2\pi h_j^{(\sigma)}, z_j^{(\sigma)} - 3/2\pi h_j^{(\sigma)} \right]
\]
 of $\mb{R}$ such that the $k^{th}$ derivative  $ \left(\psi_j^{(\sigma)}\right)^{(k)}$ of $\psi_j^{(\sigma)}$ satisfies:
\[
\left| \left(\psi_j^{(\sigma)}\right)^{(k)}\right| \leq  C \left(h^{(\sigma)}_j\right)^{-k}.
\]
\end{lemma}

\begin{proof}
First we show that for $\ol{\lambda}$ sufficiently small, the intervals $I^{(\sigma)''}_{j}$ are a covering of $\mb{R}$ and that  $I^{(\sigma)''}_{j}$ and $I^{(\sigma)''}_{k}$ intersect non-trivially if and only if $|j - k| \leq 1$. We first compare the right endpoint $z^{(\sigma)}_{j, +} : = z^{(\sigma)}_{j} + \frac{3 \pi}{2} h_j^{(\sigma)}$ of $I^{(\sigma)''}_{j}$ with the left endpoint $z^{(\sigma)}_{j + 1, -}: = z^{(\sigma)}_{j + 1}  - \frac{3 \pi } {2} h^{(\sigma)}_{j + 1}$ of $I^{(\sigma)''}_{j + 1}$: We have
\begin{align*}
z^{(\sigma)}_{j, +} - z^{(\sigma)}_{j + 1, -} & = z^{(\sigma)}_{j} + \frac{3 \pi}{2} h_j^{(\sigma)} - z^{(\sigma)}_{j + 1}  + \frac{3 \pi } {2} h^{(\sigma)}_{j + 1}\\
 & = z^{(\sigma)}_{j} + \frac{3 \pi}{2} h_j^{(\sigma)} - z^{(\sigma)}_{j} - 2 \pi h^{(\sigma)}_{j}  + \frac{3 \pi } {2} h^{(\sigma)}_{j + 1}  \\
& = - \frac{\pi}{2} h_j^{(\sigma)} + \frac{3 \pi } {2} h^{(\sigma)}_{j + 1}  \\
& =  \pi h_j^{(\sigma)} + \frac{3 \pi } {2} \left(\frac{h^{(\sigma)}_{j + 1}}{h^{(\sigma)}_{j}}- 1\right) h^{(\sigma)}_j
\end{align*}
As in the proof of Corollary \ref{Prop:DomainIntersection} we can take $\ol{\lambda}$ sufficiently small so that  the ratio
\[
\left|\frac{h^{(\sigma)}_{j + 1}}{h^{(\sigma)}_{j}}- 1\right| \leq \frac{1}{6},
\]
which then gives
\begin{align}\label{IntersectionLength}
z^{(\sigma)}_{j, +} - z^{(\sigma)}_{j + 1, -} > \frac{3\pi}{4} h_j^{(\sigma)}.
\end{align}
We similarly  have 
\begin{align*}
z^{(\sigma)}_{j, +} - z^{(\sigma)}_{j + 1} & =  z^{(\sigma)}_{j} + \frac{3 \pi}{2} h_j^{(\sigma)} - z^{(\sigma)}_{j + 1}   \\
& = z_{j}^{(\sigma)}   + \frac{3 \pi}{2} h_j^{(\sigma)} - z_j^{(\sigma)} - 2\pi h^{(\sigma)}_{j} \\
& = - \frac{\pi}{2} h_{j}^{(\sigma)},
\end{align*}
and
\begin{align*}
z^{(\sigma)}_{j + 1, -} - z_{j} & =  \left(z^{(\sigma)}_{j + 1} - \frac{3 \pi }{2} h^{(\sigma)}_{j + 1} \right) - z^{(\sigma)}_{j} \\
& =  \left( z^{(\sigma)}_{j} + 2 \pi h^{(\sigma)}_{j} - \frac{3 \pi}{2} h^{(\sigma)}_{j + 1} \right) - z^{(\sigma)}_{j} \\
& = \frac{\pi}{2} h^{(\sigma)}_j - \frac{3 \pi}{2}\left(\frac{h^{(\sigma)}_{j + 1}}{h^{(\sigma)}_{j}} - 1\right)h^{(\sigma)}_{j}
\end{align*}
and thus as above
\[
z^{(\sigma)}_{j + 1, -} - z^{(\sigma)}_{j} > \frac{\pi}{4} h_j^{(\sigma)}.
\]
Thus we have the inequalities
\[
z^{(\sigma)}_{j} < z^{(\sigma)}_{j + 1, -} < z^{(\sigma)}_{j , +} < z^{(\sigma)}_{j + 1}.
\]
In particular, this implies that the intervals $I^{(\sigma)}_{j}$ and  $I^{(\sigma)}_{k}$ are disjoint unless $k = j - 1, j, j + 1$, and it holds that 
\[
I^{(\sigma)}_{j} \cap I^{(\sigma)}_{j + 1} = \left[z^{(\sigma)}_{j + 1, -}, z^{(\sigma)}_{j, +}\right].
\]
We can then define $\psi^{(\sigma)}_{j}$ in the usual way. Away from the intersections  $I^{(\sigma)}_{j} \cap I^{(\sigma)}_{j \pm 1}$ we take $\psi^{(\sigma)}_j \equiv 1$ and on $I^{(\sigma)}_{j} \cap I^{(\sigma)}_{j + 1}$  we define $\psi^{(\sigma)}_j$ as follows:  Fix a smooth monotonic odd function $f: \mb{R}\rightarrow \mb{R}$ such that $f(x) \equiv 1/2$ for $x > \frac{1}{2}$ and  identify the interval  $I^{(\sigma)}_{j} \cap I^{(\sigma)}_{j + 1}$ with $[-l, l]$ for $l > 0$. By (\ref{IntersectionLength}) we have $\ell > c h_j^{(\sigma)}$ for fixed $c > 0$.  For $x \in [-l, l] $ we then take
\[
\psi^{(\sigma)}_{j} (x) : = 1/2 - f(x/\ell)
\]
Similarly, for $x \in [-l, l]$ we take
\[
\psi^{(\sigma)}_{j + 1} (x) =  1/2 + f(x/\ell).
\]
Clearly the functions $\psi^{(\sigma)}_j$ constitute a partition of unity of $\mb{R}$ satisfying the derivative estimate in the statement of the Lemma. 
\end{proof}

\begin{lemma} \label{Prop:CProjectors}
It holds that  
\[
\wt{\mc{L}} \left(\cos(z)\cosh (s)\right) = 2 \cos (z) \cosh^{-1}(s)
\]
and
\[
\wt{\mc{L}} \left(\sin(z)\cosh (s) \right)= 2 \sin (z) \cosh^{-1}(s).
\]
\end{lemma}
\begin{proof}
This is an easy computation.
\end{proof}

\begin{definition}  \label{Def:CAJBJs}
With  $E(s, z) = e (z)\cosh^{-1}(s)$ and $E_j : = \psi^{(\sigma)}_j E : = e_j (z)\cosh^{-1}(s) $, we set
\[
a_j : = \int_{\Lambda^{(\sigma)''}_j} E_j(s, z) \cos(z)\cosh^{-1}(s), \quad b_{j} : =  \int_{\Lambda^{(\sigma)''}_j} E_j(s, z) \sin(z) \cosh^{-1}(s).
\]
where $\Lambda^{(\sigma)''}_j : = \Lambda \cap \left(I_j^{(\sigma)''} \times \mb{R}\right)$
\end{definition}

\begin{lemma} \label{Prop:CAJBJs}
It holds that  
\[
|a_j|, |b_j|\leq^C \beta \omega^{(\sigma)}_j h^{(\sigma)}_j,
\]
where above we have set $\omega_j^{(\sigma)} : = \omega\left(z^{(\sigma)}_j \right)$.
\end{lemma}

\begin{proof}
We have
\begin{align*}
a_j  & = \int_{\Lambda^{(\sigma)''}_j}  E_j(s, z) \cosh^{-1} (s) \\
& \leq^C \| e_j\|_{0,  \alpha} \left|I^{(\sigma)}_j\right| \int_{\ol{\ell}_j}^{\ol{\ell}_j} \cosh^{-2} (s)ds \\
& \leq^C \beta\omega^{(\sigma)}_j h^{(\sigma)}_j. \\
\end{align*}
The estimate for $b_j$ follows from similar reasoning. 
\end{proof}

\begin{lemma} \label{Prop:CAJBJs2}
There is a function $\wh{\Theta}_j: \mb{R}^2 \rightarrow \mb{R}$  supported on $I^{(\sigma)'}_{j} \times \mb{R}$ satisfying the estimates
\[
\left\|\wh{\Theta}_j \right\|_{2, \alpha} \leq^C \beta \omega^{(\sigma)}_j\cosh(s), \quad \left\|\wt{\mc{L}}\wh{\Theta}_j \right\|_{0, \alpha} \leq^C \beta \omega^{(\sigma)}_j \left(\cosh^{-1}(s) +   \left(h_j^{(\sigma)}\right)^{-1}\cosh(s) \right)
\]
and such that 
\[
\int_{\Lambda_j^{(\sigma)}} \left(\wt{\mc{L}} \wh{\Theta}_j\right) \sin(z) \cosh^{-1}(s) = a_j, \quad \int_{\Lambda_j^{(\sigma)}}\left( \wt{\mc{L}} \wh{\Theta}_j\right) \cos(z) \cosh^{-1}(s) = b_j
\]
\end{lemma}
\begin{proof}
Set 
\[
A_j : = a_j/2\int_{R^{(\sigma)}_j} \varphi_j^{(\sigma)} \sin^2(z) \cosh^{-2} (s)
\]
 and 
 \[
 B_j : = b_j/2\int_{R^{(\sigma)}_j} \varphi^{(\sigma)}_j \cos^2(z) \cosh^{-2} (s) .
 \]
 and for the moment take $\wh{\Theta}_j$ to be defined by 
\[
\wh{\Theta}_j : = \sum_j A_j \varphi^{(\sigma)}_j \sin(z)\cosh(s) + \sum_j B_j \varphi^{(\sigma)}_j \cos(z)\cosh(s).
\]
We will show that $\wh{\Theta}_j$ is  an approximate solution and that an exact solution can be found by iteration. From Lemma \ref{Prop:CAJBJs} we have
\begin{align}\label{AJBJCOSH}
\left|A_j \right| \leq^C a_j/h^{(\sigma)}_j \leq^C \beta \omega^{(\sigma)}_j, \quad  \left|B_j \right| \leq^C b_j/h^{(\sigma)}_j \leq^C \beta \omega^{(\sigma)}_j,
\end{align}
and thus $\wh{\Theta}_j$ satisfies the bounds in the statement of the Lemma. Moreover, using the  estimate (\ref{VarphiSigmaEst}) for derivatives of  $\varphi^{(\sigma)}_j$ we have
\begin{align} \label{StuffNeededLater}
\wt{\mc{L}}\left(\varphi^{(\sigma)}_j \sin(z) \cosh(s)\right) & = 2 \varphi^{(\sigma)}_j \sin(z) \cosh^{-1} (s)+ \left(\varphi^{(\sigma)}_j\right)'' \sin(z) \cosh(s)  \\ \notag
& \quad + 2 \left(\varphi^{(\sigma)}_j\right)' \cos(z) \cosh(s) \\
& =  2\varphi^{(\sigma)}_j \sin(z) \cosh^{-1}(s) + O\left(\cosh(s)/h^{(\sigma)}_j\right), \notag
\end{align}
and a similar estimate holds for $\wt{\mc{L}}\left(\varphi^{(\sigma)}_j \cos(z) \cosh(s)\right)$. This and the estimate (\ref{AJBJCOSH}) for the coefficients $A_j$ and $B_j$,  together imply the estimate for $\wt{\mc{L}} \wh{\Theta}_j$  and  $\wh{\Theta}_j$  in the statement of the Lemma. Using (\ref{StuffNeededLater}) we have
\begin{align*}
\int_{R_j^{(\sigma)}} A_j \wt{\mc{L}}\left(\varphi^{(\sigma)}_j \sin(z) \cosh(s)\right)  \sin(z) \cosh^{-1}(s) &  = a_j +  A_j \int_{R_j^{(\sigma)}} O\left(\cosh(s)/h^{(\sigma)}_j\right)\sin(z) \cosh^{-1}(s) \\
& = a_j + O\left(a_j \ol{\ell}^{(\sigma)}_j/h^{(\sigma)}_j\right). \\
\end{align*}
Similarly, we have
\begin{align*}
\int_{R_j^{(\sigma)}} B_j \wt{\mc{L}}\left(\varphi^{(\sigma)}_j \cos(z) \cosh(s)\right)  \sin(z) \cosh^{-1}(s) & = 2 B_j\int_{R_j^{(\sigma)}} \varphi_j^{(\sigma)} \cos(z) \sin(z) \cosh^{-2}(s) \\
& \quad +  B_j \int_{R_j^{(\sigma)}}O\left(\cosh(s)/h^{(\sigma)}_j\right)\sin(z) \cosh^{-1}(s) \\
& = 2 B_j \int_{R_j^{(\sigma)}} \varphi_j^{(\sigma)} \cos(z) \sin(z) \cosh^{-2}(s)+  O\left(b_j \ol{\ell}^{(\sigma)}_j/h^{(\sigma)}_j\right).
\end{align*}
We have
\begin{align*}
\int_{R_j^{(\sigma)}} \varphi_j^{(\sigma)} \cos(z) \sin(z)\cosh^{-2}(s) &  \leq^C \int_{I^{(\sigma)}_j} \varphi_j^{(\sigma)}(z) \cos(z) \sin(z) dz  \\
& = -\frac{1}{2} \int_{I^{(\sigma)}_j} \left(\varphi_j^{(\sigma)}\right)'(z) \sin^{2}(z) dz \\
& = O\left(h^{(\sigma)}_j\right)^{-1}\left|I^{(\sigma)}_j \right| \\
& = O(1).
\end{align*}
and thus
\begin{align*}
\int_{R_j^{(\sigma)}} B_j \wt{\mc{L}}\left(\varphi^{(\sigma)}_j \cos(z) \cosh(s)\right)  \sin(z) \cosh^{-1}(s) &  = O (B_j) +   O\left(b_j \ol{\ell}^{(\sigma)}_j/h_j^{(\sigma)}\right) \\
& = O\left(b_j \ol{\ell}^{(\sigma)}_j/h_j^{(\sigma)}\right).
\end{align*}
Combining gives
\[
\int_{R^{(\sigma)}_j}\wt{\mc{L}} \wh{\Theta}_j \sin(z) \cosh^{-1}(s) =  a_j + O \left(\left(|a_j| + |b_j|\right)\ol{\ell}_j^{(\sigma)}/h^{(\sigma)}_j\right) =: a_j + a'_j 
\]
where $a_j'$ is implicitly defined above. Similar reasoning gives
\[
\int_{R^{(\sigma)}_j}\wt{\mc{L}} \wh{\Theta}_j \cos(z) \cosh^{-1}(s) =  b_j + O \left(\left(|a_j| + |b_j|\right)\ol{\ell}_j^{(\sigma)}/h^{(\sigma)}_j\right) = : b_j + b_j'
\]
where again $b_j'$ is implicitly defined above. Taking $\sigma > 2$ and assuming $\ol{\lambda}$ is sufficiently small we can assume  the ratio $\ell^{(\sigma)}_j/h^{(\sigma)}_j = \left(\ell^{(\sigma)}_j\right)^{1 - \sigma}$ to be arbitrarily small. In particular, we can assume that the inequality
\[
|a'_j|, |b'_j| \leq \delta \left(\left|a_j\right|  + \left|b_j\right|\right)
\]
holds for any $\delta> 0$. The process described above can then be iterated to find constants $A_j$ and $B_j$ such that 
\[
\int_{R^{(\sigma)}_j}  \left(\wt{\mc{L}}\wh{\Theta}_j \right)\sin(z) \cosh^{-1}(s)   = a_j,  \quad  \int_{R^{(\sigma)}_j}  \left(\wt{\mc{L}}\wh{\Theta}_j \right)(z - z_j) \cos(z) \cosh^{-1}(s) = b_j.
\]
It remains to estimate the integral over $\Lambda^{(\sigma)}_j$. We will use a similar iteration argument to the one presented above.  We have
\begin{align*}
\int_{\Lambda^{(\sigma)}_j} \left( \wt{\mc{L}}\wh{\Theta}_j\right) \sin(z) \cosh^{-1}(s) & =  \int_{R^{(\sigma)}_j} \left( \wt{\mc{L}}\wh{\Theta}_j\right) \sin(z) \cosh^{-1}(s) - \int_{R^{(\sigma)}_j \setminus \Lambda^{(\sigma)}_j} \left( \wt{\mc{L}}\wh{\Theta}_j\right) \sin(z) \cosh^{-1}(s) \\
& = a_j -  \int_{R^{(\sigma)}_j \setminus \Lambda^{(\sigma)}_j} \left( \wt{\mc{L}}\wh{\Theta}_j\right) \sin(z) \cosh^{-1}(s) \\
& = :  a_j + a'_j
\end{align*}
where $a'_j$ is determined implicitly above. Using the derivative estimate (\ref{EllDIFFIN}) for $\ell$ we can estimate  $a_j'$ above by:
\begin{align*}
 a_j' & = \int_{I^{(\sigma)}_j} \int_{\ell}^{\ol{\ell}^{(\sigma)}_j} \left( \wt{\mc{L}}\wh{\Theta}_j\right) \cos(z) \cosh^{-1}(s) ds dz \\
 & \leq \sup\left| \left( \wt{\mc{L}}\wh{\Theta}_j\right)  \cosh^{-1}(s) \right| \left| I^{(\sigma)}_j\right| \left|  \ell^{(\sigma)}_j - \ell(z)\right| \\
 & \leq^C  \left(|A_j| + |B_j| \right)\left( 1 + \left(h^{(\sigma)}_{j}\right)^{-1}\right)  \left(\lambda^{(\sigma)}_j\right)^{(1 - \tau) \epsilon_0} \left(h^{(\sigma)}_{j}\right)^{2} \\
 & \leq^C \left(|a_j|/h_j^{(\sigma)} + |b_j|/h_{j}^{(\sigma)} \right) \left(\lambda^{(\sigma)}_j\right)^{(1 - \tau ) \epsilon_0} \left(h^{(\sigma)}_{j}\right)^{2}\\
 &  \leq^C \left(|a_j| + |b_j| \right)  \left(\lambda^{(\sigma)}_j\right)^{(1 - \tau) \epsilon_0}\left(h^{(\sigma)}_{j}\right) \\
 & \leq^C \left(|a_j| + |b_j| \right)  \left(\lambda^{(\sigma)}_j\right)^{(1 - \tau) \epsilon_0}\left(\ell^{(\sigma)}_{j}\right)^{\sigma}
\end{align*}
Assuming $\ol{\lambda}$ sufficiently small we have 
\[
a'_j \leq \delta \left(|a_j|  + |b_j|\right).
\]
Similarly defining $b'_j$ by
\[
b_j' : = - \int_{R^{(\sigma)}_j \setminus \Lambda^{(\sigma)}_j} \left( \wt{\mc{L}}\wh{\Theta}_j\right) \cos(z) \cosh^{-1}(s)
\]
we have the estimate
\[
b'_j \leq \delta \left(|a_j|  + |b_j|\right).
\]
As before, we can thus iterate away the error terms to produce an exact solution $\wh{\Theta}_j$ as in the statement of the Lemma. 

\end{proof}

\begin{definition} \label{Def:CHAtSolution}
We set
\[
\wh{E}_j : = E - \mc{\wt{L}}\left(\wh{\Theta}_j\right).
\]
\end{definition}

\begin{lemma} \label{Prop:CHatSolution}
The following statements hold:
\begin{enumerate}
\item  \label{Prop:CHatSolution2} $\left\| \wh{E}_j\right\|_{0, \alpha}\leq^C\beta  \omega^{(\sigma)}_j \left(\cosh^{-1}(s) +\left(h^{(\sigma)}_j\right)^{- 1}\cosh(s)\right)$. \\
\item  \label{Prop:CHatSolution3} $\int_{\Lambda^{(\sigma)}_j} \wh{E}_j \sin(z) \cosh^{-1}(s) = \int_{\Lambda^{(\sigma)}_j} \wh{E}_j  \cos(z) \cosh^{-1}(s) = 0.$ \\
\end{enumerate}
\end{lemma}

\begin{proof}
Both   (\ref{Prop:CHatSolution2}) are (\ref{Prop:CHatSolution3}) immediate consequences of the definition of  $\wh{E}_j$  and Lemma \ref{Prop:CAJBJs2}. 
\end{proof}

\begin{definition} \label{Def:CBarSolutions}
We set
\begin{enumerate}
\item  \label{Def:CBarSolutions1}
\[
\theta_j(z) : = \int_{-\ell(z)}^{\ell(z)} \wh{E}_j(s, z)  \cosh^{-1}(s)ds/\int_{-\ell^{(\sigma)}_j}^{\ell^{(\sigma)}_j}\cosh^{-2}(s) ds.
\]
\item  \label{Def:CBarSolutions2}
\begin{align*}
\ol{\Theta}_j(z)  & = -\cos(z) \int_{-\infty}^{z} \theta_j(z')  \sin(z') dz' + \sin(z) \int_{-\infty}^{z} \theta_j(z') \cos(z') dz'.
\end{align*}

\end{enumerate}
\end{definition}

\begin{lemma} \label{Prop:CBarSolutions}
The functions $\theta_j$ and $\ol{\Theta}_j$ are supported on $I^{(\sigma)''}_j$ and it holds that:
\begin{enumerate}
\item \label{Prop:CBarSolutions1} $\left\|\theta_j\right\|_{0, \alpha} \leq^C \beta \omega^{(\sigma)}_j$.  \\
\item \label{Prop:CBarSolutions2} $\left\|\ol{\Theta}_j\right\|_{2, \alpha} \leq^C \beta h_j^{(\sigma)}\omega^{(\sigma)}_j$.
\end{enumerate}
\end{lemma}

\begin{proof}
For  (\ref{Prop:CBarSolutions1}) we have from Lemma \ref{Prop:CHatSolution} (\ref{Prop:CHatSolution2}) that
\begin{align*}
\left\|\theta_j\right\|_{0, \alpha} \leq^C & \beta \omega^{(\sigma)}_j  \left(\int_{- \ell (z)}^{\ell(z)}  \left(\cosh^{-1}(s) + \cosh(s)/h^{(\sigma)}_j\right)  \cosh^{-1}(s)ds\right) \\
& \leq^C  \beta \omega^{(\sigma)}_j \left(1 + \ell(z)/h^{(\sigma)}_j\right) \\
& \leq^C \beta \omega^{(\sigma)}_j
\end{align*}
Moreover, since $\wh{E}_j$ is supported on $I^{(\sigma)''}_j \times \mb{R}$ it follows directly that $\theta_j$ is supported on $I_j''$. Observe that since $\wh{E}_j$ is supported on $\Lambda^{(\sigma)''}_j$  it follows that  $\ol{\Theta}_j(z)$ vanishes for $z < z^{(\sigma)}_j - \frac{3 \pi}{2} h_j^{(\sigma)} $. Moreover, the orthogonality condition in Lemma \ref{Prop:CHatSolution} (\ref{Prop:CHatSolution3}) on $\wh{E}_j$ gives  for $z > z^{(\sigma)}_j + \frac{3 \pi}{2} h_j^{(\sigma)}$ that
\begin{align*}
\ol{\Theta}_j\left(z \right)  & = - \cos(z) \int_{- \infty}^{z} \theta_j (z')  \sin(z') dz'  +   \sin (z) \int_{-\infty}^{ z }  \theta_j (z')  \cos(z') dz'  \\ 
& = \frac{ - \cos(z)}{\int_{-\ell^{(\sigma)}_j}^{\ell^{(\sigma)}_j}\cosh^{-2}(s) ds}\int_{I^{(\sigma)''}_j} \int_{- \ell(z)}^{\ell(z)} \wh{E}_j(s, z) \cosh^{-1}(s)  \sin(z') ds dz'  \\
& \quad +  \frac{\sin(z)}{\int_{-\ell^{(\sigma)}_j}^{\ell^{(\sigma)}_j}\cosh^{-2}(s) ds} \int_{I^{(\sigma)''}_j} \int_{-\ell (z)}^{\ell(z)} \wh{E}_j(s, z) \cosh^{-1}(s)  \cos(z') ds dz'  \\
& = - \frac{\cos(z)}{\int_{-\ell^{(\sigma)}_j}^{\ell^{(\sigma)}_j}\cosh^{-2}(s) ds} \int_{\Lambda_j^{(\sigma)''}}  \wh{E}_j(s, z) \cosh^{-1}(s)  \sin(z)   \\
& \quad +  \frac{\sin(z)}{\int_{-\ell^{(\sigma)}_j}^{\ell^{(\sigma)}_j}\cosh^{-2}(s) ds} \int_{\Lambda^{(\sigma)''}_j} \wh{E}_j(s, z) \cosh^{-1}(s)  \cos(z)   \\ 
&= 0.
\end{align*}
  Thus $\Theta_j(z)$ is supported on  $I^{(\sigma)''}_j$. For $z \in I^{(\sigma)''}_j$ we have
\begin{align*}
\left\|\Theta_j(z)  \right\|_{2, \alpha} & \leq^C  \left| I^{(\sigma)''}_j \right|\left\|\theta_j \right\|_{0, \alpha} \\
& \leq^C\beta h^{(\sigma)}_j \omega^{(\sigma)}_j. 
\end{align*}
\end{proof}

\begin{proof}[Proof of Proposition \ref{Prop:CoshProjectors}]
We set $\Theta_j =   \ol{\Theta}_j \cosh^{-1}(s) + \wh{\Theta}_j$ and $\Theta : = \sum_j \Theta_j$. Since $\wt{\mc{L}}\cosh^{-1}(s) = \cosh^{-1}(s)$, we have
\[
\mc{\wt{L}}\left( \ol{\Theta}_j \cosh^{-1}(s)\right) = \left(\ol{\Theta}_j'' + \ol{\Theta}_j \right) \cosh^{-1}(s) = \theta_j \cosh^{-1}(s) 
\]
and thus 
\begin{align*}
\int_{-\ell(z)}^{\ell(z)} \wt{\mc{L}}\Theta_j \cosh^{-1}(s) ds & = \int_{-\ell(z)}^{\ell(z)} \wt{\mc{L}}\wh{\Theta}_j \cosh^{-1}(s) ds +   \int_{-\ell(z)}^{\ell(z)} \wt{\mc{L}}\left(\ol{\Theta}_j\cosh^{-1}(s)\right) \cosh^{-1}(s) ds \\
& = \int_{-\ell(z)}^{\ell(z)} \left(E_j - \wh{E}_j\right) \cosh^{-1}(s) ds + \int_{-\ell(z)}^{\ell(z)}  \theta_j \cosh^{-2}(s) ds \\
& =  \int_{-\ell(z)}^{\ell(z)} \left(E_j - \wh{E}_j\right) \cosh^{-1}(s) ds  +\left(\frac{ \int_{-\ell(z)}^{\ell(z)} \cosh^{-2}(s) ds}{ \int_{-\ell^{(\sigma)}_j}^{\ell^{(\sigma)}_j} \cosh^{-2}(s) ds} \right)\int_{-\ell(z)}^{\ell(z)} \wh{E}_j \cosh^{-1}(s)ds \\
& = \int_{-\ell(z)}^{\ell(z)} E_j(s, z) \cosh^{-1}(s) ds  +\left(\frac{ \int_{-\ell(z)}^{\ell(z)} \cosh^{-2}(s) ds}{ \int_{-\ell^{(\sigma)}_j}^{\ell^{(\sigma)}_j} \cosh^{-2}(s) ds} - 1 \right)\int_{-\ell(z)}^{\ell(z)} \wh{E}_j \cosh^{-1}(s)ds\\
& =  e_j(z)\int_{-\ell(z)}^{\ell(z)} \cosh^{-2}(s) ds + e'_j \int_{-\ell(z)}^{\ell(z)} \cosh^{-2}(s) ds.
\end{align*}
where above we have put $e_j : = \psi^{(\sigma)}_j e$ and where $e'_j$ is defined implicitly above by the last equality.  We have
\begin{align*}
\left\| \int_{-\ell(z)}^{\ell(z)} \cosh^{-2}(s) ds -  \int_{-\ell^{(\sigma)}_j}^{\ell^{(\sigma)}_j} \cosh^{-2}(s) ds \right\|_{0, \alpha} &\leq C \left\| \ell(z) - \ell^{(\sigma)}_j \right\|_{0, \alpha} \cosh^{-2} 
\left(\ell^{(\sigma)}_j\right) \\
& \leq C \left(\lambda^{(1 - \tau) \epsilon}_j h_j^{(\sigma)} \cosh^{-2}\left(\ell^{(\sigma)}_j\right) \right)
\end{align*}
and thus
\begin{align*}
\left\| e'_j \right\|_{0, \alpha} \leq C  \beta \omega^{(\sigma)}_j\left( \lambda^{(1 - \tau) \epsilon}_j h_j^{(\sigma)} \cosh^{-2}(\ell_j) \right) \left(1 + \left(h^{(\sigma)}_j\right)^{-1}  \ell_j^{(\sigma)}\right)
\end{align*}
By Lemma \ref{Prop:EllBounds} we can take $\ol{\lambda}$ depending on $\ol{\sigma}$ such that 
\[
\left\|e'_j \right\|_{0, \alpha} \leq \delta \beta \omega^{(\sigma)}_j
\]
for arbitrary $\delta > 0$.  With $e' : = \sum_j e'_j$, we have from the above estimate, Remark \ref{EllComparable} and the fact the supports of $e'_j$ and $e'_k$ intersect only if $|k - j| \leq 1$ that 
\[
\left\| e' \right\|_{0, \alpha} \leq C \delta \beta \omega.
\]
We can then iterate to find an exact solution $\Theta$ satisfying the claimed bound in the statement of Proposition \ref{Prop:CoshProjectors}.
\end{proof}

\section{Proof of Theorem \ref{MainInvertibilityStatement}} \label{Sec:MainInvertibilityProof}
 We first record Lemma \ref{ExtensionLemma}, which constructs the kernel density preserving $C^{0, \alpha}$ extension map.

 \begin{lemma} \label{ExtensionLemma}
There is a bounded linear map $E \mapsto \check{E}$ from $C^{0, \alpha} (\Lambda) \rightarrow C^{0, \alpha} (\check{\Lambda})$ and such that 
\begin{enumerate}
\item $\check{E}$ is  supported away from the boundary of $\check{\Lambda}$.
\item $\check{E} (s, z) = E(s, z)$ for $(s, z) \in \Lambda$. 
\item The map $E \mapsto \check{E}$ preserves the $\tanh(s)$ and $\cosh(s)$ densities. That is, it holds that
\[
\int_{- \ell }^{\ell} E(s, z) \cosh^{-1}(s)ds = \int_{- \ell - 1}^{\ell + 1} \check{E}(s, z)  \cosh^{-1}(s)ds
\]
and
\[
\int_{- \ell}^{\ell} E(s, z) \tanh(s)ds = \int_{- \ell - 1}^{\ell + 1} \check{E}(s, z)  \tanh(s)ds
\]
\end{enumerate}
\end{lemma}
\begin{proof}
Fix a smooth monotonically decreasing function $\varphi: \mb{R} \rightarrow \mb{R}$  such that $\varphi (s) \equiv 1$ for $s \leq 0$ and $\varphi = 0$ for $s \geq 1/2$.  For  $s \geq \ell(z)$ we define $\check{E}$ by
\[
\check{E}(s, z) = \varphi(s - \ell(z) )\left( a_+(z)( s - \ell(z)) + E(2 \ell(z)  - s, z) \right)
\]
and for  $s \leq - \ell(z)$ we set
\[
\check{E}(s, z) = \left(\varphi(- \ell(z) - s )\right)\left( a_-(z)( -s - \ell) + E(- 2 \ell(z)  - s, z) \right)
\]
for constants $a_\pm(z)$ to be determined. We have
\begin{align*}
\int_{\ell}^{\ell+ 1}\check{E}(s, z) \tanh(s) ds & =\int_{\ell}^{\ell + 1} \varphi(s - \ell)  E(2 \ell(z) - s, z) \tanh(s) + a_+\int^{\ell + 1}_{\ell} \tanh (s) (s - \ell) \varphi (s -\ell) ds,\\ \\
\int_{\ell}^{\ell+ 1}\check{E}(s, z) \cosh^{-1}(s) ds & = \int_{\ell}^{\ell + 1} \varphi(s - \ell)  E(2 \ell(z) - s, z) \cosh^{-1}(s)   + a_+ \int^{\ell + 1}_{\ell} \cosh^{-1} (s) (s - \ell) \varphi (s - \ell) ds,
\end{align*}
and
\begin{align*}
\int_{-\ell - 1}^{-\ell}\check{E}(s, z) \tanh(s) ds & = \int_{-\ell - 1}^{- \ell} \varphi(-\ell - s)  E(-2 \ell(z) - s, z) \tanh(s)  -  a_-(z)\int_{-\ell - 1}^{-\ell } \varphi(- s - \ell) (s +   \ell(z)) \tanh (s)
\\ &  =   \int_{-\ell - 1}^{- \ell} \varphi(-\ell - s)  E(-2 \ell(z) - s, z) \tanh(s)  -  a_-(z)\int_{\ell}^{\ell + 1} \tanh(s) (s - \ell)\varphi( s - \ell)  ds,
 \\ \\
\int_{-\ell - 1}^{-\ell}\check{E}(s, z) \cosh^{-1}(s) ds & = \int_{-\ell - 1}^{- \ell} \varphi(-\ell - s)  E(-2 \ell(z) - s, z) \cosh^{-1}(s)  - a_-(z)\int_{-\ell - 1}^{-\ell } \varphi(- s - \ell) (s +  \ell(z)) \cosh^{-1}(s) \\
&=  \int_{-\ell - 1}^{- \ell} \varphi(-\ell - s)  E(-2 \ell(z) - s, z) \cosh^{-1}(s)  + a_-(z)\int_{\ell}^{\ell + 1 } \varphi(s - \ell) (s +  \ell(z)) \cosh^{-1}(s).
\end{align*}
Set: 
\[
A_+ : = \int_{\ell}^{\ell + 1} \varphi(s - \ell)  E(2 \ell(z) - s, z) \tanh(s) ds, \quad A_- : = \int_{-\ell - 1}^{- \ell} \varphi(-s - \ell)  E(-2 \ell(z) - s, z) \tanh(s)ds
\]
Assuming the first and third equations sum to zero we get
\[
a_+ - a_- = -\frac{A_+ + A_-}{\int_{\ell}^{\ell + 1 }(s - \ell(z) \varphi(s - \ell)  \tanh(s) ds }   = :  R.
\]
Clearly we have that 
\[
\| R\|_{0, \alpha} \leq^C \| E\|_{0, \alpha}.
\]
Also, write
\begin{align*}
B_+  & : =  \int_{\ell}^{\ell + 1} \varphi(s - \ell)  E(2 \ell(z) - s, z) \cosh^{-1} (s) \\
& = \cosh^{-1}(\ell) \int_{\ell}^{\ell + 1} \varphi(s - \ell)  E(2 \ell(z) - s, z) \cosh_{\ell}^{-1}(s) \\
& =   \cosh^{-1}(\ell)  \wt{B}_+,
\end{align*}
where above we have set $\cosh_{\ell}(s) = \cosh(s)/\cosh(\ell)$ and  similarly put
\begin{align*}
B_-  & : =  \int_{-\ell - 1}^{- \ell} \varphi(-s - \ell)  E(-2 \ell(z) - s, z) \cosh^{-1}(s)   \\
& = \cosh^{-1}(\ell)  \int_{-\ell - 1}^{- \ell} \varphi(-s - \ell)  E(-2 \ell(z) - s, z) \cosh_{\ell}^{-1} (s) \\
& =   \cosh^{-1}(\ell)  \wt{B}_-,
\end{align*}
Then assuming  the second and fourth equations sum to zero gives
\begin{align*}
a_+ +  a_- & =- \frac{B_+  + B_-}{\int_{\ell}^{\ell + 1 } \varphi(s - \ell) (s - \ell(z)) \cosh^{-1} (s)}\\
& =  \frac{\wt{B}_+  + \wt{B}_-}{\int_{\ell}^{\ell + 1 } \varphi(s - \ell) (s - \ell(z)) \cosh_{\ell}^{-1}(s) } \\
& =: Q,
\end{align*}
where $Q$ is defined implicitly by the last equality above. As before we have
\[
\| Q\|_{0, \alpha} \leq^C \| E\|_{0, \alpha}.
\]
Solving for $a_-$ and $a_+$ then gives 
\[
\| a_\pm\|_{0, \alpha} \leq^C \| E\|_{0, \alpha}.
\]
The extension $\check{E}$ then satisfies
\[
\| \check{E}\|_{0, \alpha} \leq^C \| E\|_{0, \alpha} .
\]
as well as the remaining claims. 
\end{proof}

\subsection{Proof of Theorem \ref{MainInvertibilityStatement}}  \label{Sec:SummingSolutions}

Let $E: \Lambda \rightarrow \mb{R}$ be a locally $C^{0, \alpha}$ function on $\Lambda$ with
\[
\left\| E \right\|_{0, \alpha} \leq \beta \omega.
\]
By Lemma \ref{ExtensionLemma} we can assume that $E$ is supported away from the boundary of $\check{\Lambda}$. With $E_{j} : = \psi_j E$, we have   $\left\| E_j\right\|_{0, \alpha} \leq^C \beta \omega_j$. By Theorem \ref{RBCSolutionsWeightedControl}, there is a function $v_j :[- \ol{\ell}_j^*- 1, \ol{\ell}^*_j + 1] \times \mb{R} \rightarrow \mb{R}$ with
\[
\wt{\mc{L}} v_j = E_j
\]
and satisfying the weighted estimate
\begin{align*}
\left\| v_j  \right\|_{2, \alpha} &  \leq^C \left(\ol{\ell}^*_j\right)^{5/2} \beta \omega_j  \left( \frac{1}{1 + |z - z_j|}\right)^{\xi}.\\
\end{align*}
Define the  functions $v_j^*$ and $E_j^*$  by:
\begin{align*}
v^*_j (s, z): = \psi_j^*(z)v_j (s, z), \quad E_j^* : = E_j - \tilde{\mathcal{L}} (\psi_j^*v_j).
\end{align*}
We then have
\begin{lemma} \label{DualObjectProperties}
The following statements hold:
\begin{enumerate} \notag
\item \label{LocalizedSupports}The functions $v^*_j$ and $E^*_j$  are supported on $R_j^*$ \\
\item \label{DualSolutionEstimate} The functions $v^*_j$ satisfy the estimates
\begin{align}
\left\| v^*_j \right\|_{2, \alpha} \leq   C\beta \omega_j \ol{\ell}^{5/2}_j \left(\frac{1}{1 + |z - z_j|} \right)^{\xi}.
\end{align}

\item \label{DualErrorEstimate} The functions $E^*_j$ satisfies the estimate
\begin{align}
\left\| E^*_j \right\|_{0, \alpha} \leq   C\beta \omega_j\ol{\ell}^{5/2 - \sigma}_j \left(\frac{1}{1 + |z - z_j|} \right)^{\xi}.
\end{align}

\item \label{LocalESTARESTIMATE} We have
\begin{enumerate}
\item 
\begin{align*}
\left|\int_{- \ell}^{\ell}E^*_j (s, z)\tanh(s) \right| & \leq^C \beta \frac{\ol{\ell}^{5/2 - \sigma}_j\omega_j}{1 + |z - z_j|^{\xi}}.
\end{align*} 
\item
\begin{align*}
 \left|\int_{- \ell}^{\ell}E^*_j (s, z)\cosh^{-1}(s) \right| & \leq^C \beta \frac{\ol{\ell}^{5/2 - \sigma}_j\omega_j \cosh^{-1} (\ell)}{1 + |z - z_j|^{\xi}}.
\end{align*}
\end{enumerate}
\end{enumerate}
\end{lemma}

\begin{proof}
Statement (\ref{LocalizedSupports}) is a direct consequence of the definition of $\psi^*_j$ in Section \ref{POFUABF}.  Statements  (\ref{DualSolutionEstimate}) and  (\ref{DualErrorEstimate}) follow directly for the weighted estimate for $v_j$, the definition of $\psi^*_j$ and the weighted estimate for solutions in the local strongly orthogonal case in Proposition \ref{RBCSolutionsWeightedControl}. For (\ref{LocalESTARESTIMATE}), observe that
\[
\int_{- \ol{\ell}_j - 1}^{\ol{\ell}_j + 1} E_j^*(s, z) \tanh(s) ds = 0
\]
and thus
\begin{align*}
\int_{\ell}^{\ell} E_j^*(s, z) \tanh(s) & \leq^C \| E^*_j\|_{0, \alpha} |\ell - (\ol{\ell}_j + 1)| \\
& \leq^C \beta \frac{\omega_j \ol{\ell}^{5/2- \sigma }_{j}}{1 + |z - z_j|^{\xi}} \\
\end{align*}
 Similarly, we have
\begin{align*}
\int_{\ell}^{\ell} E_j^*(s, z) \cosh^{-1}(s) ds &  \leq C \| E^*_j\|_{0, \alpha}\cosh^{-1} (\ell) |\ell - (\ol{\ell}_j + 1)| \\
& = \beta\frac{\omega_j \ol{\ell}^{5/2 -  \sigma}_{j} }{1 + |z - z_j|^{\xi}} \cosh^{-1} (\ell).
\end{align*}

\end{proof}
\begin{lemma} \label{FirstIterationFunctions}
Set 
\begin{align} \notag
v^* : = \sum_j v^*_j, \quad E^* : = \sum_j E^*_j.
\end{align}
Then the following statements hold:
\begin{enumerate}
\item \label{LocallyFiniteSums} The infinite sums defining $v^*$ and $E^*$ are locally finite and thus converge on compact subsets of $\Lambda$. 
\item \label{VStarEstimate} The function $v^*$ satisfies the estimate
\begin{align} \notag
\left\| v^*\right\|_{2, \alpha} \leq C \beta \ell^{5/2 + (1 - \xi) \sigma}\omega.
\end{align}
\item \label{EStarEstimate} The function $E^*$ satisfies the estimate
\begin{align} \notag
\left\| E^*\right\|_{0, \alpha} \leq C \beta \ol{\ell}^{5/2 - \xi \sigma}  \omega.
\end{align}

\item \label{ESTARDENSITY}
It holds that 
\[
\left|\int_{-\ell}^{\ell}E^*(s, z)\cosh^{-1}(s) \right| \leq^C  \beta \ell^{5/2 - \xi \sigma }\omega \cosh^{-1} (\ell),  \quad \left|\int_{-\ell}^{\ell}E^*(s, z)\tanh(s) \right| \leq^C  \beta\ell^{5/2 - \sigma \xi}\omega.
\]
\end{enumerate}
\end{lemma}
\begin{proof}
Fix a $j \in \mathbb{N}$.  By Corollary \ref{Prop:DomainIntersection}, the supports of $v^*_k$ intersect $\lambda_j$ only if $j - k < 2 h^*_j$. Thus on   $\Lambda_j$ the sums defining $v^*$ are finite and given by. 
\begin{align} \notag
\left\|\sum_{k} v^*_k \right\|_{2,\alpha}& = \sum_{ |k -j| \leq  2 h^*_j}\left\| v^*_k \right\|_{2, \alpha} \\ \notag
 & \leq  C \beta \sum_{ |k - j| \leq  2h^*_j} \frac{ \ol{\ell}^{5/2}_{j}\omega_k}{(1 + |z_k - z_j|)^{\xi}} \\ \notag
 & \leq C \beta \ol{\ell}^{5/2}_{j} \omega_j  \int_{- 4\pi h^*_j}^{ 4 \pi h^*_j} \left(\frac{1}{1 + |w| } \right)^{\xi }dw \\ \notag
 & \leq  C  \beta\ol{\ell}^{5/2}_j\omega_j (h^*_j)^{1 - \xi }\\ \notag
  & \leq  C  \beta\ol{\ell}^{5/2}_j\omega_j (\ell^{\sigma}_j)^{1 - \xi }\\ \notag
  & \leq C \beta \ell^{5/2 + (1 - \xi) \sigma} \omega
\end{align} 
This   gives  (\ref{VStarEstimate}). Claim (\ref{EStarEstimate}) follows similarly:  On $\Lambda_j$ we  again have
\begin{align*} \notag
\sum_{k} E^*_k & = \sum_{k :  |k - j| \leq  2 h^*_j} E^*_k \\ 
&  \leq C\beta   \sum_{k:  |k - j| \leq 2 h^*_j} \frac{ \ol{\ell}_k^{5/2 -\sigma }\omega_k}{\left(1 + |z_k - z_j|\right)^{\xi}} \\ 
& \leq C\beta\ol{\ell}_j^{5/2- \sigma }\omega_j\int_{- 4\pi  h^*_j}^{4 \pi h^*_j}  \frac{dw}{\left(1 + w\right)^{\xi} } \\ 
& \leq  C\beta \ol{\ell}_j^{5/2 - \sigma}  \omega_j \left(h^*_{j}\right)^{1 - \xi}\\
& \leq^C \beta \ol{\ell}_j^{5/2 - \sigma} \left(\ell^{\sigma}_j\right)^{1-\xi} \omega_j \\
&  \leq^C \beta \ol{\ell}_j^{5/2 - \xi \sigma}\omega_j.
\end{align*}
  For (\ref{ESTARDENSITY}),  we have as above
\begin{align*}
\int_{-\ell}^{\ell} E^* \tanh(s) & = \sum_{k: |j-  k| \leq 2 h^*_j} \int_{-\ell}^\ell E_k^* \tanh(s) ds \\
& \leq^C\beta \ol{\ell}_j^{5/2 - \sigma }\omega_j  \sum_{k: |j-  k| \leq 2 h^*_j}   \frac{1}{1 + (|z_k - z_j|)^{\xi}} \\
& \leq^C \beta  \ol{\ell}_j^{5/2 - \xi \sigma }\omega_j
\end{align*}
Similarly we have
\[
\int_{-\ell}^{\ell} E^* \cosh^{-1}(s)  \leq^C \beta\ol{\ell}_j^{5/2 - \xi \sigma }\omega_j\cosh^{-1} (\ell).
\]

\end{proof}

Theorem \ref{MainInvertibilityStatement} now follows immediately.
\begin{proof}[Proof of Theorem \ref{MainInvertibilityStatement}]
Let $E : = E_0$ be a function as in the statement of the proposition with
\begin{align} \notag
\| E \|_{0, \alpha} \leq \beta \omega.
\end{align}
We then apply Lemma \ref{FirstIterationFunctions} to obtain functions $v^*$, $E^*$ satisfying the following estimates:
\begin{align} \notag
\| v^* \|_{2,\alpha} \leq C \beta \ell^{5/2 + (1 - \xi) \sigma } \omega, \quad  \| E^*\|_{2, \alpha} \leq C \beta \ell^{5/2 - \xi \sigma} \omega.
\end{align}
We will assume that $\sigma$ is chosen sufficiently large so that $5/2  - \xi \sigma  <0$. 
Let $e_T$ and $e_C$  denote the $\tanh(s)$ and $\cosh^{-1}(s)$ densities of $E^*$, respectively, so:
\[
e_T : = \frac{\int_{-\ell}^{\ell}E^*(s, z)\tanh(s)ds}{\int_{\ell}^\ell \tanh^2(s)}, \quad e_C   : = \frac{\int_{-\ell}^{\ell}E^*(s, z)\cosh^{-1}(s)ds}{\int_{\ell}^\ell \cosh^{-2}(s)}.
\]
By Lemma \ref{FirstIterationFunctions} (\ref{ESTARDENSITY}) we have
\[
\| e_T\|_{0, \alpha} \leq^C \beta \ell^{5/2 - \xi \sigma } \omega , \quad \| e_C\|_{0, \alpha}  \leq \beta \ell^{5/2 - \xi\sigma} \omega\cosh^{-1} (\ell).
\]
We claim  that  the weight functions  $\wt{\omega}_{T} : = \ell^{5/2 - \xi\sigma} \omega$ and  $\wt{\omega}_{C}  = \ell^{5/2 - \xi\sigma}  \cosh^{-1}(\ell) \omega$   satisfy (\ref{WeightFunctionCondition}) and are thus admissible weight fucntions. To see this, pick $z$ and $z_0$ with $|z - z_0| < \lambda^{- 3 \epsilon_0/4}(z_0)$. Then with $p = 5/2 - \xi \sigma$ we have
\begin{align*}
\left|\frac{\wt{\omega}_T(z)}{\wt{\omega}_T (z_0)} - 1\right|  & =\left| \left(\frac{\ell(z)}{\ell(z_0)}\right)^{p}\frac{\omega(z)}{\omega(z_0)}  - 1\right|\\
&= \left| \left(\left(\frac{\ell(z)}{\ell(z_0)}\right)^{p} - 1\right)\frac{\omega(z)}{\omega(z_0)} + \frac{\omega(z)}{\omega(z_0)}  - 1\right|  \\
& \leq  \left|\left(\frac{\ell(z)}{\ell(z_0)}\right)^{p} - 1\right| \left(1 + \rho_0 \right)  +  \rho_0 \\
\end{align*}
With $\lambda_0 : = \lambda(z_0)$ we have using the derivative estimate (\ref{EllDIFFIN}) for $\ell_\tau$ that
\begin{align*}
\ell(z) - \ell(z_0) & \leq C \lambda_0^{(1 - \tau) \epsilon_0} \lambda_0^{-3/4 \epsilon_0} \\
& \leq C \lambda_0^{(1/4 - \tau) \epsilon_0}
\end{align*}
and thus
\begin{align*}
\left|\frac{\ell(z)}{\ell(z_0)} - 1\right| & = \left|\frac{\ell(z) - \ell(z_0)}{ \ell(z_0)} \right|\\
& \leq C \ell^{-1}(z_0) \lambda_0^{(1/4 - \tau) \epsilon_0} \\
& \leq C \ell^{-1} (z_0) \lambda_0^{\epsilon_0/8}
\end{align*}
 assuming $\ol{\tau} < 1/8$. Given $q > 0$ we  can  then take $\ol{\lambda}$ sufficiently small so that 
 \[
 \left|\frac{\ell(z)}{\ell(z_0)} - 1\right| < q
 \]
 which gives
 \begin{align*}
 \left|\frac{\wt{\omega}_T(z)}{\wt{\omega}_T (z_0)} - 1\right|  & \leq \left(\left( 1 + q\right)^{p} -1 \right)(1 + \rho_0) + \rho_0 \\
 & \leq C| p| \log(1 + q) \left( 1 + \rho_0 \right) + \rho_0 \\
 & \leq C \ol{\sigma} q  \left( 1 + \rho_0 \right) + \rho_0
 \end{align*}
 where above we have used the that $p < 0$ and thus $|p| \leq \xi \sigma <  \ol{\sigma}$ since $\xi \in (0, 1)$.
 Taking $q$ small, we then have
 \[
  \left|\frac{\wt{\omega}_T(z)}{\wt{\omega}_T (z_0)} - 1\right| \leq \wt{\rho}_0 = \frac{1 + \rho_0}{2}.
 \]
For the weight function $\wt{\omega}_C$ we write
\begin{align*}
\frac{\cosh^{-1} \left(\ell(z)\right)}{\cosh^{-1}\left(\ell(z_0) \right)} & \approx e^{\ell_0 - \ell} \\
& \approx \exp\left(C\lambda^{(1 -  \tau) \epsilon_0}(z_0) \lambda^{- 3/4 \epsilon_0}(z_0) \right) \\
& \approx \exp \left(C\lambda^{( 1/4 -  \tau) \epsilon_0}(z_0) \right) \\
& \lesssim\exp \left(C\lambda^{( 1/8) \epsilon_0}(z_0) \right)
\end{align*}
and thus similar reasoning gives
\[
\left|\frac{\wt{\omega}_C(z)}{\wt{\omega}_C(z_0)} - 1 \right| \leq \wt{\rho}_0 < 1
\]
for $\ol{\lambda}$ sufficiently small.  We can thus apply   Propositions \ref{Prop:TanhProjectors} and \ref{Prop:CoshProjectors} with weights $\wt{\omega}_{T}$ and $\wt{\omega}_{C}$, respectively,   to get functions $\Theta_T$ and $\Theta_C$ such that 
\[
\int_{- \ell}^\ell \wt{\mc{L}}\Theta_T = e_T\int_{\ell}^{\ell} \tanh^2(s) ds, \quad \int_{- \ell}^\ell \wt{\mc{L}}\Theta_C = e_C\int_{\ell}^{\ell} \cosh^{-2}(s) ds
\]
and the estimates
\begin{align*}
\left\| \Theta_T\right\|_{2, \alpha}&  \leq^C  \beta \ell^{9/2 - \xi\sigma}\omega, \\  \\
 \left\| \Theta_C\right\|_{2, \alpha}&  \leq^C  \beta \ell^{5/2 - \xi \sigma}\omega \cosh^{-1} (\ell)  \cosh(s), \\ \\
\left\|\wt{\mc{L}} \Theta_T \right\|_{0, \alpha}  & \leq^C   \beta \ell^{9/2 - \xi\sigma}\omega, \\ \\
\left\|\wt{\mc{L}} \Theta_C \right\|_{0, \alpha} & \leq^C \beta \ell^{5/2 - \xi \sigma}\cosh^{-1} (\ell)
(\cosh^{-1}(s) +  \ell^{- \sigma} \cosh(s)) \\
& \leq^C  \beta \ell^{5/2 - (1 + \xi )\sigma},
\\
\end{align*}
where in the last line above we have assume $\ol{\lambda}$ is sufficiently small so that $\cosh^{-1}(\ell) < \ell^{- 1} < \ell^{- \sigma}$.
With $\Theta = \Theta_T + \Theta_C$, the function 
\[
E_1 : = E^* - \wt{\mc{L}}\Theta.
\]
Then satisfies the estimate
\begin{align}\label{E1Estimate}
\| E_1\| & \leq \| E^*\| + \| \wt{\mc{L}} \Theta \| \\
& \leq^C  \beta \ell^{5/2 - \xi \sigma} \omega \notag
 + \beta \ell^{9/2 - \xi\sigma}\omega  + \beta \ell^{5/2 - (1 + \xi) \sigma} \omega \\ \notag
& \leq^C \beta \omega\left(  \ell^{5/2 - \xi \sigma}
 +  \ell^{9/2 - \xi\sigma}  + \ell^{5/2 - (1 + \xi) \sigma}  \right). \notag
\end{align}
We take $\sigma$ large enough so that the exponents of $\ell$ appearing on the  right hand side above are negative. Then, taking $\ol{\lambda}$ sufficiently small, we have
\[
\| E_1 \|_{0, \alpha} \leq \delta  \beta \omega
\]
for any arbitrary $\delta > 0$.
Setting $v_0 : = v^* + \Theta$ we have
\begin{align}\label{V0EstimateII}
\| v_0\|_{2, \alpha} & \leq^C \beta \ell^{5/2 + (1 - \xi) \sigma } \omega
 + \beta \ell^{9/2 - \xi\sigma}\omega + \beta \ell^{5/2 - \xi \sigma}\omega \cosh^{-1} (\ell)  \cosh(s) \\\notag
& \leq^C \beta \omega\left(\ell^{5/2 + (1 - \xi) \sigma } 
 + \beta \ell^{9/2 - \xi\sigma} + \beta \ell^{5/2 - \xi \sigma} \cosh^{-1} (\ell)  \cosh(s) \right) \\ \notag
 & \leq C \beta \ell^{5/2 + (1 - \xi) \sigma}\left(1 + \ell^{2 - \sigma} + \ell^{- \sigma} \right) \notag
\end{align}
Thus, assuming $\sigma > 3$ we have the estimate
\begin{align} \label{V0Estimate}
\|v_0\|_{2, \alpha}  \leq^C\beta  \ell^{5/2  + (1-  \xi) \sigma } \omega.
\end{align}
We then iterate to obtain an exact solution satisfying estimate (\ref{V0Estimate}). At this point, we make explicit choices for the parameters $\sigma$  and $\xi$. The parameter $\xi$ can be chosen freely in $(0, 1)$ and we take $\xi = \frac{1}{4}$.  We then chose $\sigma =  20$,   so that the exponents of $\ell$ appearing in (\ref{E1Estimate}) and (\ref{V0EstimateII}) are negative.  This then proves the assertions of Theorem \ref{MainInvertibilityStatement} with $c  = 5/2 + (1 - \xi) \sigma = 5/2+ (1 - 1/4) 20 = 5/2 + 15 = 35/2$.

\end{proof}

 \bibliographystyle{amsalpha}

\end{document}